\documentclass[final,10pt]{amsart}
\usepackage{amsmath,amsfonts,amssymb,amscd,verbatim}
\usepackage{fullpage}
\usepackage{enumerate}
\usepackage{bm}
\usepackage{xcolor}
\usepackage{enumitem}
\usepackage{tikz-cd}
\usepackage{mathtools}
\usepackage[all]{xy}

\numberwithin{figure}{section}
\numberwithin{table}{section}
\numberwithin{equation}{section}
\theoremstyle{plain}
\newtheorem{thm}{Theorem} 
\newtheorem{theorem}[equation]{Theorem} 

\newtheorem{corollary}[equation]{Corollary} 
\newtheorem{lemma}[equation]{Lemma}
\newtheorem{proposition}[equation]{Proposition}

\theoremstyle{definition}
\newtheorem{definition}[equation]{Definition} 
 
\newtheorem{example}[equation]{Example}
 
\newtheorem{question}[equation]{Question}
\newtheorem{remark}[equation]{Remark}

\newtheorem*{theorem*}{Theorem}

\DeclareMathOperator\Aut{Aut}

\DeclareMathOperator\diag{diag}

\DeclareMathOperator\Ext{Ext}
\DeclareMathOperator\Frac{Frac}
\DeclareMathOperator\GL{GL}
\DeclareMathOperator\gldim{gldim}
\DeclareMathOperator\gr{gr}

\DeclareMathOperator\lcm{lcm}

\DeclareMathOperator\rank{rank}

\DeclareMathOperator\SL{SL}
\DeclareMathOperator\Span{Span}
\DeclareMathOperator\Tr{Tr}
\DeclareMathOperator\trace{trace}

\newcommand\NN{\mathbb N}

\newcommand\ZZ{\mathbb Z}

\newcommand\cF{\mathcal F}

\newcommand\cO{\mathcal O}
\newcommand\cS{\mathcal S}

\newcommand\bp{\mathbf p}
\newcommand\bq{\mathbf q}

\newcommand\cnt{\mathcal Z}
\newcommand\detq{\det\nolimits_q}
\renewcommand{\int}{\mathrm{int}}
\newcommand\inv{^{-1}}

\newcommand\iso{\cong}

\newcommand\tensor{\otimes}
\newcommand\tornado{\xi}

\newcommand\grp[1]{\langle #1 \rangle}

\renewcommand{\to}{\ensuremath{\longrightarrow}}

\newcommand\gu{\underline g}
\newcommand\cu{\underline \chi}

\newcommand\kk{\Bbbk}
\newcommand\qp{\kk_{\mu}[u,v]}
\newcommand\wa{A_1^\mu(\kk)}

\newcommand\qmn{\cO_q(M_N(\kk))}
\newcommand\qgl{\cO_q(\GL_N(\kk))}
\newcommand\qsl{\cO_q(\SL_N(\kk))}

\newcommand{\wanp}{A_t^{\bp,\gamma}(\kk)}

\newcommand{\ol}{\overline}
\newcommand{\ul}{\underline}
\newcommand{\wh}{\widehat}
\newcommand{\mc}{\mathcal}
\newcommand{\ord}{\textup{ord}}
\newcommand{\defeq}{\vcentcolon=}

\newenvironment{smatrix}
  {\left(\begin{smallmatrix}}
  {\end{smallmatrix}\right)}

\makeatletter 
\let\c@table\c@figure 
\makeatother

\title{Actions of quantum linear spaces on quantum algebras}

\author[Cline]{Zachary Cline}
\address{Bucknell University, Department of Mathematics, Lewisburg, Pennsylvania 17837}
\email{z.cline@bucknell.edu}

\author[Gaddis]{Jason Gaddis}
\address{Miami University, Department of Mathematics, 301 S. Patterson Ave., Oxford, Ohio 45056} 
\email{gaddisj@miamioh.edu}


\begin{document}

\begin{abstract}
We study actions of bosonizations of quantum linear spaces on quantum algebras. Under mild conditions, we classify actions on quantum affine spaces and quantum matrix algebras.
In the former case, it is shown that all actions of generalized Taft algebras are trivial extensions of actions on quantum planes.
In both cases we achieve bounds on the rank of the bosonization acting on the algebra. 
\end{abstract}

\maketitle

\section{Introduction}

Due to the principle of quantum rigidity, quantum algebras exhibit few classical symmetries, i.e., linear group actions.
For example, the automorphism group of the quantum plane $\kk_q[u,v]$ with $q^2 \neq 1$ is isomorphic to $(\kk^\times)^2$ \cite{AC}.
Here quantum algebras will not take on a specific meaning, but will generally be understood to represent some algebra whose relations depend on parameters in $\kk$.
This includes quantum affine spaces and quantum matrix algebras, both of which are fundamental objects in the study of noncommutative algebra and noncommutative algebraic geometry.

In many cases, the (graded/filtered) automorphism group for quantum algebras are known, see, e.g., \cite{AC,CPWZ2,y-llc}.
The natural next step, then, is to study quantum symmetries, or actions by Hopf algebras.
Semisimple Hopf actions on quantum planes and quantum Weyl algebras are well-understood \cite{CWWZ,CKWZ}.
Our goal is to better understand non-semisimple Hopf actions,
specifically actions by pointed Hopf algebras, which themselves have attracted much recent interest
\cite{Cl,KW}.

The impetus for this work was a classification by Won, Yee, and the second-named author of Taft algebra actions on quantum planes and quantum Weyl algebras \cite{GWY}.
Here we ask how much this classification can be extended.
We do this in several ways.
First, we look at actions of {\it generalized} Taft algebras and find that the classification problem is not significantly different.
Secondly, we consider actions on higher-dimensional algebras, specifically quantum affine spaces and quantum matrix algebras.
Finally, we study actions of bosonizations of quantum linear spaces (see \cite{AS2,EW}).

Bosonizations of quantum linear spaces form an important subclass within the classification of finite-dimensional pointed Hopf algebras of Andruskiewitsch and Schneider \cite{AS1}. In some sense, they may be thought of as higher rank generalized Taft algebras.
Under mild hypotheses --- in particular, we require all parameters have order greater than 2 --- we classify actions of generalized Taft algebras on quantum affine spaces and quantum matrix algebras.
This is then extended to determine all actions of bosonizations of quantum linear spaces, again under mild hypotheses.
Specifically, we achieve bounds on the rank of these bosonizations.
It is our hope that our methods may be applied for further classifications and a long term goal is to understand the classification of {\it all} finite-dimensional pointed Hopf algebras on these algebras.

\section{Preliminaries}

Throughout, $\kk$ is an algebraically closed, characteristic zero field and all algebras are associative $\kk$-algebras. All unadorned tensor products should be regarded as over $\kk$. For a Hopf algebra $H$, and grouplike elements $g,h \in G(H)$, we denote by $P_{g,h}$ the $(g,h)$-skew-primitive elements, i.e. all $x \in H$ so that $\Delta(x) = g \otimes x + x \otimes h$.

An algebra $A$ is {\sf ($\NN$)-graded} if there exists a vector space decomposition $A = \bigoplus_{i \in \NN} A_{(i)}$ such that ${A_{(i)} \cdot A_{(j)} \subset A_{(i+j)}}$.
Further, $A$ is {\sf connected} if $A_{(0)} = \kk$, {\sf affine} if $A_{(k)}$ is finite-dimensional as a $\kk$-vector space for all $k$, and {\sf generated in degree one} if $A_{(1)}$ generates $A$ as an algebra.
With the exception of one family in Section \ref{sec.add}, all algebras considered in this work are affine connected graded and generated in degree one.

\subsection{Quantum algebras}
Here we define more formally our algebras of interest.
Our standard reference is \cite{BGquantum}.

A matrix $\bp = (p_{ij}) \in M_t(\kk^\times)$ is {\sf multiplicatively antisymmetric} 
if $p_{ii}=1$ and $p_{ij}=p_{ji}\inv$ for all $i,j$.
Given a multiplicatively antisymmetric matrix $\bp$, the 
{\sf quantum affine space}, denoted $A=\kk_\bp[u_1,\hdots,u_t]$, is generated by $u_1,\hdots,u_t$ subject to the relations $u_iu_j=p_{ij}u_ju_i$.
If $I \subset \{1,\hdots,t\}$, then we denote by $A_I$ the subalgebra of $A$ generated by the $u_i$, $i \in I$.
If $I=\{i,j\}$, then we simply write $A_{ij}$ for the {\sf quantum plane} $\kk_{p_{ij}}[u_i,u_j]$.

Let $q \in \kk^\times$.
  We denote by $\cO_q (M_n(\kk))$ the {\sf single parameter quantum $n \times n$ matrix algebra}.
  It is generated by $Y_{ij}$ for $1 \leq i, j \leq n$ with relations
  \[
    Y_{i j} Y_{\ell m} = 
      \begin{cases}
        q Y_{\ell m} Y_{ij}, & (i < \ell, \ j = m) \\
        q Y_{\ell m} Y_{ij}, & (i = \ell, \ j < m) \\
        Y_{\ell m} Y_{ij}, & (i < \ell, \ j > m) \\
        Y_{\ell m} Y_{ij} + (q - q^{-1}) Y_{im} Y_{\ell j}, & (i < \ell, \ j < m).
      \end{cases}
  \]
  For $\cO_q(M_2(\kk))$, we will use the traditional notation $A = Y_{1,1}, \ B = Y_{1,2}, \ C = Y_{2,1}, \ D = Y_{2,2}$.
  Thus, $\cO_q(M_2(\kk))$ has relations
  \[
    AB = qBA, \quad
    AC = qCA, \quad
    BD = qDB, \quad
    CD = qDC, \quad
    BC = CB, \quad
    AD = DA + (q - q^{-1}) BC.
  \]

By \cite[Theorem 5.2]{EW2}, under certain conditions on $\bp$, the action of a finite-dimensional Hopf algebra on $\kk_\bp[u_1,\hdots,u_t]$ factors through a group action. In particular, this holds when $p_{ij}=q \in \kk^\times$ for all $i>j$ (the single parameter case) and $q$ is not a root of unity. That is, these algebras possess no true finite-dimensional quantum symmetry.
Similarly, since $\Frac(\cO_q(M_N(\kk)))$ is isomorphic to a $\kk_\bp[u_1,\hdots,u_{N^2}]$ for suitable $\bp$ \cite{cliff}, then a similar result holds in this setting. Part of the current work is to better understand the quantum symmetries when the parameters fall outside of these conditions.

\subsection{Quantum linear spaces}
Let $\theta \in \NN$, 
$G$ a finite abelian group,
$\gu = g_1,\hdots,g_\theta \in G$, and 
$\cu = \chi_1,\hdots,\chi_\theta \in \widehat{G} = \text{Hom}_{\text{(Groups)}}(G, \kk^\times)$ such that $\chi_i(g_j) = \chi_j(g_i)^{-1}$ for $i \neq j$ and $m_i \defeq \ord(\chi_i(g_i)) \geq 2$.
Let $\mc R(g_1, \ldots, g_\theta, \chi_1, \ldots, \chi_\theta)$ be a {\sf quantum linear space} over $G$,
that is, the braided Hopf algebra in ${}^{\kk G}_{\kk G} \mc{YD}$, generated by $x_1, \ldots, x_\theta$, with relations
\[	x_i^{m_i} = 0, \quad 	x_i x_j = \chi_j(g_i) x_j x_i  \ (i \neq j).\]
The coalgebra structure is determined by $x_i \in P_{1,1}$.
The comodule and module structure are determined by
\[	\delta(x_i) = g_i \otimes x_i, \quad h \cdot x_i = \chi_i(h) x_i \quad (h \in G).\]
For $g \in G$, let $\nu_g = \#\{i \mid g_i = g\}$.

The bosonization $H \defeq \mc R \# \kk G$ is the Hopf algebra (over $\kk$) generated by $G$, $x_1, \ldots, x_\theta$, 
with the relations of $G$ as well as 
\[
	x_i^{m_i} = 0, \quad
	x_i x_j = \chi_j(g_i) x_j x_i \ (i \neq j), \quad
	h x_i = \chi_i(h) x_i h \ (h \in G),
\]
with $G = G(H)$ and $x_i \in P_{g_i, 1}(H)$.
Throughout, we denote $H$ by $B(G,\gu,\cu)$ and refer to $\theta$ as its {\sf rank}.
We will primarily in this paper be concerned with actions of $B(G,\gu,\cu)$ on various families of algebras.

\begin{example}
Let $G=\ZZ_9=\grp{g}$ and let $\omega$ be a primitive $9$th root of unity.
Set $g_1=g$ and $g_2 = g^4$ and define $\chi_1,\chi_2 \in \wh G$ by $\chi_1(g)=\omega^3$ and $\chi_2(g)=\omega^6$. Then
\[ \chi_1(g_1) = \omega^3, \quad \chi_1(g_2) = \omega^3, \quad \chi_2(g_2) = \omega^6, \quad \chi_2(g_1) = \omega^6.\]
It now follows that all of the compatibility criteria are met. In particular, \[\chi_1(g_2) \chi_2(g_1) = 1, \quad \ord(\chi_1(g_1)), \ \ord(\chi_2(g_2)) \geq 2.\]
Thus, $B(G,\{g_1,g_2\},\{\chi_1,\chi_2\})$ is a rank 2 bosonization of a quantum linear space.
\end{example}

Let $\gamma \in \kk$, let $n,m \in \NN$ such that $m \mid n$, and let $\lambda \in \kk$ be a primitive $m$th root of unity.
The {\sf generalized Taft algebra} $T_n(\lambda,m,\gamma)$ is generated by a grouplike element $g$ and a $(g,1)$-skew primitive element $x$ subject to the relations
\[ g^n = 1, \quad x^m = \gamma(g^m-1), \quad gx=\lambda xg.\]
If $\chi \in \widehat{\grp{g}}$ is chosen such that $\chi(g) = \lambda$, then $T_n(\lambda,m,0) = B(\grp{g},g,\chi)$ and $T_n(\lambda,n,0)$ is the $n$th Taft algebra \cite{T}.
Given $B(G,\gu,\cu)$, we denote by $B_i$ the subalgebra of $B$ generated by $\{g_i,x_i\}$. Then $B_i \iso T_{n_i}(\lambda_i,m_i,0)$ as Hopf algebras. Here, $n_i$ denotes $\ord(g_i)$, $\lambda_i := \chi_i(g_i)$, and $m_i = \ord(\lambda_i)$.

In the $B(G,\gu,\cu)$, the $x_i$ are all nilpotent.
However, we may occasionally drop the nilpotency requirement when considering general pointed Hopf algebras of rank one.
We do not define rank here, but use the following classification of rank one pointed Hopf algebras in characteristic zero, due to Krop and Radford.

\begin{thm}[{\cite[Theorem 1]{KR}}]
Let $G$ be a finite group with character map $\chi:G \rightarrow \kk^\times$, and take $g \in \cnt(G)$ and $\gamma \in \kk$. Set $m=\ord(\chi(g))$.
Let $H(G,g,\chi,\gamma)$ denote the Hopf algebra generated by $G$ and a $(g,1)$-skew-primitive element $x$ subject to the group of relations of $G$ and the relations
\[ x^m = \gamma(g^m-1) \quad\text{and}\quad ax=\chi(a)xa\]
for all $a \in G$. Every finite-dimensional pointed Hopf algebra of rank one is isomorphic to $H(G,g,\chi,\gamma)$ for some $G,\ g,\ \chi$, and $\gamma$.
\end{thm}

\subsection{Hopf actions}
We say that a Hopf algebra $H$ {\sf acts on} an algebra $A$ if $A$ is a left $H$-module algebra.
That is, $A$ is a left $H$-module with action $h\tensor a\mapsto h\cdot a$ such that 
\[
  h \cdot 1_A = \varepsilon(h)1_A \quad \text{and} \quad
  h \cdot (aa') = \sum (h_1 \cdot a)(h_2 \cdot a') \quad
  (\text{for all } h \in H \text{ and } a, a' \in A).
\]
The action is said to be {\sf inner faithful} if there is no nonzero Hopf ideal that annihilates $A$. For Taft algebras, this is equivalent to the existence of an element $a \in A$ such that $x \cdot a\neq 0$ \cite[Lemma 2.5]{KW}. We now explore the question of inner faithfulness for actions of the $B(G,\gu,\cu)$.

If $V$ denotes the $\kk$-span of $x_1, \ldots, x_\theta$, then the action of $G$ on $V$ given by $h \cdot x_i = \chi_i(h) x_i$ is faithful if and only if for each $h \neq 1$ in $G$, there exists $i$ such that $\chi_i(h) \neq 1$.  
Let $N = \{h \in G \mid \chi_i(h) = 1 \text{ for all } i\}$ denote the kernel of the action.
Note that the induced action of $G/N$ on $V$ is faithful, and that we can realize $\mc R(g_1, \ldots, g_\theta, \chi_1, \ldots, \chi_\theta)$ as a quantum linear space over $G/N$.
  
\begin{definition}
We say that $\mc R(g_1, \ldots, g_\theta, \chi_1, \ldots, \chi_\theta)$ is a \emph{faithful} quantum linear space over $G$ if the action of $G$ on $V$ is faithful.
\end{definition}

\begin{example}
(1) Let $T=T_n(\lambda,m,0)$. 
Note that $T$ is the bosonization $\mc R(g, \chi) \# \kk G $ where $G$ is the cyclic group of order $n$ with generator $g$, $\lambda$ is a primitive $m^{th}$ root of unity, and $\chi$ is defined by $\chi(g) = \lambda$.
Thus, $\mc R(g, \chi)$ in this case is a faithful quantum linear space over $\grp{g}$ if and only if $\chi(g^\ell) \neq 1$ whenever $0 < \ell < n$. Equivalently, $\lambda^\ell \neq 1$ for $0 < \ell < n$.
Since $m=\ord(\lambda)$ divides $n$, it follows that the action is faithful if and only if $m=n$, whence $T$ is a Taft algebra.

(2)
More generally, if $G$ is a finite abelian group and $\mc R(g, \chi)$ is a quantum linear space over $G$ of rank one, then it is easy to show that $\mc R(g, \chi)$ is faithful over $G$ if and only if $G$ is cyclic and $\chi$ is a generator of $\wh G$.
\end{example}

The inner faithfulness of actions of bosonizations of faithful quantum linear spaces depends only on the actions of the $x_i$.
To show this, we require the following facts.
          
\begin{lemma}[{\cite[Lemma~1.2]{Cl}}] \label{lem:primsinideals}
Let $H$ be a pointed Hopf algebra and $I$ a nonzero Hopf ideal of $H$. 
Then $I$ contains a nonzero element of $P_{g,1}(H)$ for some $g \in G(H)$. \qed 
\end{lemma}


  \begin{lemma}[{\cite[Corollary~5.3]{AS2}}] \label{lem:qlsprims}
    In the Hopf algebra $\mc R(g_1, \ldots, g_\theta, \chi_1, \ldots, \chi_\theta) \# \kk G$, we have
    \[
      P_{g,1} = \kk(1-g) \oplus \left( \bigoplus_{i: g_i = g} \kk x_i \right).
    \]
  \end{lemma}
  
  \vspace{-.2in}
  
  \qed

  \begin{proposition}
  \label{prop.if}
    Let $\mc R(g_1, \ldots, g_\theta, \chi_1, \ldots, \chi_\theta)$ be a faithful quantum linear space over a finite abelian group $G$.
    Also, assume that for any $g \in G$ satisfying $\nu_g \geq 2$, we have $m_i \neq 2$ (i.e. $\chi_i(g_i) \neq -1$) for all $i$ such that $g_i = g$.
    Then an action of $H \defeq \mc R \# \kk G$ on some algebra $A$ is inner faithful if and only if each $x_i$ acts by nonzero.
  \end{proposition}
  
    \begin{proof}
      First, if some $x_i$ acts by zero, then the (Hopf) ideal generated by $x_i$ gives a nonzero Hopf ideal which acts by zero, so the action is not inner faithful.
      
      On the other hand, suppose the action is not inner faithful, and let $I$ denote a nonzero Hopf ideal which acts by zero.
      Let $\ol g$ and $\ol x_i$ denote the generators of $H / I$.
      By Lemma~\ref{lem:primsinideals}, there is some nonzero $a \in P_{g,1}(H) \cap I$ for some $g \in G(H)$.
      Suppose $a \in \kk (1-g)$, so $\ol g = 1$ in $H/I$. 
      Then for each $x_i$, we have 
      \[
        \ol x_i 
          = \ol g \ol x_i 
          = \chi_i(g) \ol x_i \ol g
          = \chi_i(g) \ol x_i.
      \]
      Since $\mc R$ is a faithful quantum linear space over $G$, we must have that $\chi_i(g) \neq 1$ for some $i$, and thus, $x_i \in I$.
      Hence, $x_i$ acts by zero.
      
      Now suppose that $a \notin \kk (1-g)$, so $\nu_g \geq 1$.
      Then by Lemma~\ref{lem:qlsprims}, we have two cases to check: $a = 1-g + \sum_{i: g_i = g} \alpha_i x_i$ or $a = \sum_{i: g_i = g} \alpha_i x_i$.
      We assume the former first.
      Since $\nu_g \geq 1$, we have $\alpha_i \neq 0$ for at least one $i$; without loss of generality, this is 1.
      Thus, in $I$, we have $\ol g = 1 + \sum_{i: g_i = g} \alpha_i \ol x_i$.
      Therefore, we have $\ol g \ol x_1 = \chi_1(g) \ol x_1 \ol g$, i.e. that
      \[
        \ol x_1 + \sum_{i: g_i = g} \alpha_i \ol x_i \ol x_1
          = \left(1 + \sum_{i: g_i = g} \alpha_i \ol x_i \right) \ol x_1 
          = \chi_1(g) \ol x_1 \left(1 + \sum_{i: g_i = g} \alpha_i \ol x_i \right)
          = \chi_1(g) \ol x_1 + \chi_1(g) \sum_{i: g_i = g} \alpha_i \ol x_1 \ol x_i.
      \]
      Using the fact that $x_i x_1 = \chi_1(g) x_1 x_i$ for all $i \neq 1$ such that $g_i = g$, we have
      \[
        \ol x_1 + \alpha_1 \ol x_1^2 = \chi_1(g) \ol x_1 + \chi_1(g) \alpha_1 \ol x_1^2.
      \]
      Since $\chi_1(g) \neq 1$, we have $\ol x_1 ^2 = - \alpha_1^{-1} \ol x_1$.
      Inductively, $\ol x_1^{m} = (- \alpha_1^{-1})^{m-1} \ol x_1$.
      Hence, $\ol x_1 = \ol 0$, or $x_1 \in I$.

      For the second case, we note first that $\nu_g \leq 2$.
      Otherwise, we have $\chi_1(g_1) = \chi_1(g_2) = \chi_2(g_1)^{-1} = \chi_2(g_2)^{-1}$, and similarly that $\chi_2(g_2) = \chi_3(g_3)^{-1}$ and $\chi_3(g_3) = \chi_1(g_1)^{-1}.$
      Therefore, for each $g_i = g$, we have $\chi_i(g_i) = -1$, contrary to our hypothesis. 
      Thus, we may assume $x_1 - \alpha x_2 \in I$ with $g_1 = g_2$ and $\alpha \neq 0$.
      In $H/I$, we have $\ol x_1 = \alpha \ol x_2$.
      For every $h \in G$, the relation $\ol h \ol x_1 = \chi_1(h) \ol x_1 \ol h$ yields 
      \[
        \chi_2(h) \alpha \ol x_2 \ol h 
        = \alpha \ol h \ol x_2 
        = \chi_1(h) \alpha \ol x_2 \ol h.
      \]
      If $\chi_1 = \chi_2$, then $\chi_1(g_1) = \chi_2(g_2) = \chi_1(g_1)^{-1}$, so $\chi_1(g_1) = -1$, contrary to our hypothesis.
      Therefore, for some $h \in G$, $\chi_1(h) \neq \chi_2(h)$, and so $\ol x_2 \ol h = 0$. 
      Since $\ol h$ is a unit, $\ol x_2 = 0$, or $x_2 \in I$.
    \end{proof}
    
By the above, we can always replace $G$ by a quotient so that $\mc R$ is faithful, in which case Proposition~\ref{prop.if} applies. 
For this reason, when dealing with actions of $B(G,\gu,\cu)$, we often assume merely that each $x_i$ acts by nonzero, rather than the more strict assumption that the action be inner faithful.

Whenever some $B(G,\gu,\cu)$ acts on an affine connected graded algebra $A$ that is generated in degree 1, we assume that actions are {\it linear}, that is $g_i \cdot A_{(1)}, x_i \cdot A_{(1)} \subset A_{(1)}$.
This means that, by an abuse of notation, we can represent the $g_i$ and $x_i$ as matrices, which we do throughout. We say that $g_i$ acts {\sf diagonally} on $A$ if $g_i$ is represented by a diagonal matrix.
As each $g_i$ is a grouplike in $B$, then it necessarily acts as an automorphism on $A$.

The next result, though simple, will be of great assistance in all of our classifications.

\begin{lemma}
\label{lem.gxcom}
Suppose $H(G,g,\chi,\gamma)$ acts linearly and inner faithfully on a connected graded affine algebra $A$, which is generated by $A_{(1)}=\Span_\kk\{u_1,\hdots,u_t\}$.
Assume that $g$ acts diagonally on $A_{(1)}$ and $m\geq 3$.
Also, assume the action of $x$ is linear with the action on the basis $(u_1, \ldots, u_t)$ of $A_{(1)}$ given by the matrix $(\eta_{ij})$.
Then for all $i,j$, $\eta_{ij}\eta_{ji}=0$. In particular, for all $k$, $\eta_{kk}=0$.
\end{lemma}

\begin{proof}
Set $\lambda=\chi(g)$. Let $\alpha_i \in \kk^\times$ be defined by $g \cdot u_i = \alpha_i u_i$.
A computation shows that the coefficient for $u_i$ in $(gx-\lambda xg) \cdot u_k$ is 
\begin{align}
\label{eq.coeff}
\eta_{ik}(\alpha_i-\lambda \alpha_k).
\end{align}
It follows that $0=\eta_{kk}\alpha_k(1-\lambda)$, so $\eta_{kk}=0$.
Furthermore, if for some $i\neq j$, $\eta_{ij}$ and $\eta_{ji}$ are nonzero, then
we have $\alpha_i=\lambda \alpha_j = \lambda^2 \alpha_i$, implying $\ord(\lambda) \leq 2$, 
contradicting our standing hypotheses.
\end{proof}

\subsection{Results}
Suppose a generalized Taft algebra $T=T_n(\lambda,m,0)$ acts on a quantum affine space $A=\kk_\bp[u_1,\hdots,u_t]$, $t \geq 3$.
We say the action of $T$ is a {\sf trivial extension} of the action on $A_I$ if $x \cdot u_j=0$ for all $j \notin I$.
We show in Theorem \ref{thm.nspace} that every action of a generalized Taft algebra on a quantum affine space is a trivial extension of an action on a quantum plane subalgebra, given in Proposition \ref{prop.genaction}, or a certain quantum 3-space subalgebra. This is then applied to prove the following.

\begin{theorem*}[Theorem \ref{thm.QLSrank}]
Suppose $B=B(G,\gu,\cu)$ has rank $\theta$, and that $B$ acts linearly and inner faithfully on $A=\kk_\bp[u_1,\hdots,u_t]$, $t \geq 2$.
Assume $m_i$ for all $i$ and $\ord(p_{ij})$ for all $i \neq j$ are at least 3. Then $\theta \leq 2(t-1)$.
\end{theorem*}

In Propositions \ref{prop:matrixactions} and \ref{prop:matrixactions2},  we completely classify actions of generalized Taft algebras on quantum matrix algebras under mild hypotheses. In general, the action of $x$ in this case corresponds to shifting a row or column. From this, we achieve the following result.

\begin{theorem*}[Theorems \ref{thm.mat2} and \ref{thm.matn}]
Let $q \in \kk^\times$ with $q \neq \pm 1$. Also, let $B(G,\gu, \cu)$ be a bosonization of rank $\theta$ with $m_i \geq 3$ for all $i$. Suppose $B(G,\gu, \cu)$  acts on $\mc O_q( M_N(\kk))$ with each $g_i$ acting as an element of $(\kk^\times)^{2n-1} \rtimes \langle \tau \rangle$ and each $x_i$ acting linearly and nonzero. Then, 
\[ \theta \leq \begin{cases}
    3 & \text{if $N=2$} \\
    2(N-1) & \text{if $N\geq 3$}.
    \end{cases}\]
\end{theorem*}

In Section \ref{sec.add}, we consider some peripheral results.
First, we consider invariants of generalized Taft actions on quantum planes.
Results for actions on $\kk_\bp[u_1,\hdots,u_t]$ and $\cO_q(M_N(\kk))$ are used to study actions and obtain bounds on quantum exterior algebras (Proposition \ref{prop.exterior}), quantum Weyl algebras (Proposition \ref{prop.weyl}),  as well as $\qgl$ and $\qsl$ (Proposition \ref{prop.slgl}).
We propose a number of extensions to this work and additional questions in Section \ref{sec.ques}.

\section{Quantum affine spaces}

In this section, we primarily consider actions on quantum affine spaces, and discuss the first quantum Weyl algebra.
As a warm-up, we consider the algebra $A= \kk\langle u,v \mid uv-\mu vu - \kappa\rangle$ and $\ord(\mu)=k>1$.
Then $A=\qp$ when $\kappa = 0$. When $\kappa \neq 0$, $A \iso \wa$, a {\sf quantum Weyl algebra}.
The following result is a generalization of \cite[Proposition 2.1]{GWY} to the case of generalized Taft algebra actions on quantum planes and quantum Weyl algebras.

\begin{proposition}
\label{prop.genaction}
Let $A=\qp$ or $\wa$, and let $\ord(\mu)=k>1$. 
Then $T=T_n(\lambda,m,0)$ acts linearly and inner faithfully on $A$
if and only if $n=\lcm(k,m)$ and the action is given by one of the following:
\begin{enumerate}[label=(\alph*)]
\item $g \cdot u = \mu u$, $g \cdot v = \lambda\inv\mu v$, $x \cdot u = 0$, $x \cdot v = \eta u$ for some $\eta \in \kk^\times$, and if $A= \wa$ then $\lambda = \mu^2$; or \label{list.daction}
\item $g \cdot u = \lambda\inv\mu\inv u$, $g \cdot v = \mu\inv v$, $x \cdot u = \eta v$ for some $\eta \in \kk^\times$, $x \cdot v = 0$, and if $A = \wa$ then $\lambda = \mu^{-2}$.
\end{enumerate}
\end{proposition}

\begin{proof}
By \cite{AC,AD}, either $g$ acts diagonally or anti-diagonally with respect to the given generators. There are no linear actions with $g$ acting non-diagonally on the given generators when $x \cdot A \neq 0$
and the proof of this follows similarly to \cite[Proposition 2.1]{GWY}.
Thus, we will assume that $g$ acts diagonally with respect to the given generators. 

With respect to the basis $\{u,v\}$ for $A_1$, let $x=(\eta_{ij})$ and $g=\diag(\alpha_1,\alpha_2)$ where $\alpha_i \in \kk^\times$ are $n$th roots of unity. In the case $A = \wa$ we have the additional restriction that $\alpha_2 = \alpha_1\inv$.
By Lemma~\ref{lem.gxcom}, $\eta_{11}=\eta_{22}=0$. Moreover, $\eta_{12}=0$ or $\eta_{21}=0$, but not both.

If $\eta_{21} = 0$, then $x \cdot u = 0$ and $x \cdot v = \eta_{12} u$. Furthermore,
\[ 0 = x \cdot (uv-\mu vu-\kappa) = (\alpha_1 - \mu)\eta_{12} u^2.\]
Thus, $\alpha_1 = \mu$ and so by \eqref{eq.coeff}, $\alpha_2=\lambda\inv\mu$. 
In the case of $\wa$, this implies $\lambda = \mu^2$. 
Similarly, if $\eta_{12} = 0$, then
$x \cdot u = \eta_{21}v$, $x \cdot v = 0$ and $(1-\mu\alpha_2)\eta_{21} v^2 = 0$ so $\alpha_2 = \mu \inv$ wherein $\alpha_1 = \lambda^{-1}\mu\inv$. 
In the case of $\wa$, this implies $\lambda = \mu^{-2}$. In either case, to satisfy $g^n=1$, we must have $k \mid n$.

Suppose the action of $T$ on $A$ is given as above.
Then $x \neq 0$ and so $T$ acts inner faithfully if $g$ acts faithfully on $A$. The result then follows because the order of the (matrix representation) of $g$ is $\lcm(\ord(\mu),\ord(\lambda\inv \mu)) = \lcm(\ord(\mu),\ord(\lambda)) = \lcm(k,m)$. 
\end{proof}

One of our goals will be to approach a classification along the lines of Proposition \ref{prop.genaction}
for quantum affine spaces. 
Though we do not state our classification so explicitly, we do characterize all actions
on quantum affine spaces in a way that we detail below.

By \cite[Lemma 3.5(e)]{KKZ1} and under our hypotheses, namely $p_{ij}\neq 1$, any automorphism on a quantum affine space $A=\kk_\bp[u_1,\hdots,u_t]$ may be represented by a monomial matrix.
That is, if $g \in \Aut(A)$, then there exists $\sigma_g \in \cS_t$ such that for all $k$, $g \cdot u_k = \alpha_k u_{\sigma_g(k)}$ for some $\alpha_k \in \kk^\times$.
We will show that under certain conditions we are able to limit the permutations associated to $g$.

Suppose $A$ is an algebra generated by $u_1,\hdots,u_t$. Under the linearity hypothesis, $x \cdot u_k = \sum_{i=1}^t \eta_{ik} u_i$ for all $k$.
We say $u_i$ is a {\sf summand} of $x \cdot u_k$ if $\eta_{ik} \neq 0$.
Alternatively, we say that $x \cdot u_k$ contains $u_i$ as a summand.

\begin{lemma}
\label{lem.diagonal}
Suppose $H(G,g,\chi,\gamma)$ acts linearly and inner faithfully on $A=\kk_\bp[u_1,\hdots,u_t]$, $t, \ord(p_{ij}) \geq 3$.
If $x$ is nilpotent or $m>t$, then $g$ acts diagonally on $A$.
\end{lemma}

\begin{proof}
Our goal is to show that $\ord(\sigma_g)=1$. Note that
\begin{align}
\label{eq.perm}
0   = g \cdot (u_iu_j-p_{ij}u_ju_i) = \alpha_i \alpha_j (u_{\sigma_g(i)}u_{\sigma_g(j)}-p_{ij}u_{\sigma_g(j)}u_{\sigma_g(i)})
    = \alpha_i \alpha_j (p_{\sigma_g(i)\sigma_g(j)}-p_{ij})
u_{\sigma_g(j)}u_{\sigma_g(i)}.
\end{align}

First suppose that $\sigma_g$ is a $t$-cycle.
After possibly renumbering the generators of $A$, we may assume that the action of $g$ on $A$ is defined by $g \cdot u_i = \alpha_{i+1} u_{i+1}$ for $1 \leq i < t$ and $g \cdot u_t=\alpha_1 u_1$. Set $\lambda=\chi(g)$. Then for $k<t$,
\begin{align*}
(gx-&\lambda xg) \cdot u_k
	= g \cdot (\eta_{1k}u_1 + \cdots + \eta_{tk}u_t) - \lambda x \cdot (\alpha_{k+1} u_{k+1}) \\
	&= \left( \eta_{1k}(\alpha_2 u_2) + \cdots + \eta_{(t-1)k}(\alpha_t u_t) + \eta_{tk}(\alpha_1 u_1) \right)
		- \lambda\alpha_{k+1} (\eta_{1(k+1)}u_1 + \cdots + \eta_{t(k+1)}u_t) \\
	&= (\eta_{tk}\alpha_1-\lambda\alpha_{k+1}\eta_{1(k+1)})u_1
		+ (\eta_{1k}\alpha_2-\lambda\alpha_{k+1}\eta_{2(k+1)})u_2 + \cdots
		+ (\eta_{(t-1)k}\alpha_t-\lambda\alpha_{k+1}\eta_{t(k+1)})u_t.
\end{align*}
A similar computation with $u_t$ now shows that for any $k$,
\begin{equation}\label{eq:circulant}
  \eta_{1k} = \lambda \frac {\alpha_{k+1}}{\alpha_2} \eta_{2,k+1} = \lambda^2 \frac {\alpha_{k+1}}{\alpha_2} \frac {\alpha_{k+2}}{\alpha_3} \eta_{3, k+2} = \ldots
\end{equation}
where subscripts are understood $(\mathrm{mod}~t)+1$.
Thus, entries along skew diagonals of $x$ are either all zero or all nonzero.
For $1<j<n$,
\[ x \cdot  (u_1u_j-p_{1j}u_ju_1)
	= \left[ (\alpha_2 u_2)(x \cdot u_j) + (x \cdot u_1)u_j\right] - p_{1j} \left[ (\alpha_{j+1} u_{j+1})(x \cdot u_1) + (x \cdot u_j)u_1\right].\]
When $j=n$ the same computation holds but $\alpha_{j+1} u_{j+1}$ is replaced by $\alpha_1 u_1$. Since by \eqref{eq.coeff} $u_1$ is not a summand of $x \cdot u_1$, then it is clear that $u_1^2$ appears as a summand only in the product $(x \cdot u_j)u_1$. As $p_{1j} \neq 0$, it must be that $\eta_{1j} = 0$.
It follows from \eqref{eq:circulant} that $x$ is represented by the diagonal matrix $\diag(a,\lambda\inv a,\cdots,\lambda^{-(t-1)} a)$ with $\ord(\lambda) \leq t$. Such a matrix is nilpotent if and only if $x=0$.

Next assume that $\sigma_g = (1~2~\cdots~k)$ for some $1<k < t$.
If $k=2$, then by \eqref{eq.perm}, $p_{12}=p_{21}=p_{12}\inv$, contradicting our hypothesis on the $p_{ij}$.
Hence, we may assume that $k>2$ and also that $t>2$.
The proof above shows that the upper-left $k \times k$ block of $x$ will be a diagonal matrix of the form
$\diag(a,\lambda\inv a,\cdots,\lambda^{-(k-1)} a)$ with $\ord(\lambda) \leq k<t$. Let $i\leq k$ and $k+1 \leq j \leq t$. Then
\begin{align}
\label{eq.cyc}
x \cdot (u_iu_j - p_{ij}u_ju_i)
	= ( (g \cdot u_i)(x \cdot u_j) + (x \cdot u_i)u_j) - p_{ij}( (g \cdot u_j)(x \cdot u_i) + (x \cdot u_j)u_i).
\end{align}
Since $u_i$ is not a summand of $g \cdot u_i$ or $g \cdot u_j$ it follows that the coefficient of $u_i^2$ is $-p_{ij}\eta_{ij}$, so $\eta_{ij}=0$.
Now $x$ is represented by a block lower-triangular matrix where the upper left block is the stated diagonal matrix, whence $x$ is not nilpotent.

Finally, we assume that $\sigma_g = \tau_1 \cdots \tau_\ell$ for disjoint nontrivial cycles $\tau_i$.
After possibly renumbering the generators, 
write $\tau_1=(1~2~\cdots~k)$, $\tau_2=(k+1~k+2~\cdots~k+k')$, and so on.
We partition $x$ into blocks $(X_{ij})$ where $X_{ii}$ is a $\ord(\tau_i) \times \ord(\tau_i)$ matrix.
The arguments above show that $X_{ii}$ will be diagonal matrices and that $\ord(\lambda) \leq \ord(\tau_i)$ for each $i$, $1 \leq i \leq \ell$.
Moreover, the argument following \eqref{eq.cyc} shows that $X_{ij}=0$ for $i<j$. But then $x$ is not nilpotent.
\end{proof}

In light of Lemma \ref{lem.diagonal}, we assume henceforth that $g$ acts diagonally on a quantum affine space, so $g \cdot u_i=\alpha_i u_i$ for some $\alpha_i \in \kk$ and $x \cdot u_j=\sum \eta_{ij} u_i$.
The next lemma shows that the possible actions of $x$ are limited.

\begin{lemma}
\label{lem.xprops}
Suppose $H(G,g,\chi,\gamma)$ acts linearly and inner faithfully on $A=\kk_\bp[u_1,\hdots,u_t]$.
Assume $m$, $t$, and $\ord(p_{ij})$ for $i \neq j$ are all at least $3$,
and that $g$ acts diagonally on $A$.
\begin{enumerate}
\item For all $i,j$, $\eta_{ij}\eta_{ji}=0$. In particular, for all $k$, $\eta_{kk}=0$.
\item There is at most one nonzero entry in each column of $x$.
\item There is at most one nonzero entry in each row of $x$.
\item If $t>3$ and $i,j,k,\ell$ are all distinct, then $\eta_{ij}\eta_{k\ell} = 0$.
\item If $m\neq 3$, then the matrix $x$ is nilpotent.
\item If $\ord(g)=m$ or $t \leq m \neq 3$, then $\gamma=0$.
\end{enumerate}
\end{lemma}

\begin{proof}
(1) This is Lemma \ref{lem.gxcom}.

(2) 
Suppose that $\eta_{rk},\eta_{sk} \neq 0$, with $r<s$. 
By (1), $r,s \neq k$. Then
\begin{align*}
0 	&= x \cdot (u_ru_k-p_{rk}u_ku_r) \\
	&= \left( \alpha_r u_r \left(\sum_{j \neq k} \eta_{jk} u_j\right) + \left(\sum_{i \neq r} \eta_{ir} u_i\right)u_k\right)
		- p_{rk} \left( \alpha_k u_k \left(\sum_{i \neq r} \eta_{ir} u_i\right) + \left(\sum_{j \neq k} \eta_{jk} u_j\right)u_r\right) \\
	&= \eta_{rk} (\alpha_r-p_{rk}) u_r^2 + \eta_{sk} (\alpha_r p_{rs} - p_{rk})u_s u_r + \text{(terms not involving $u_r^2$ and $u_su_r$)}.
\end{align*}
Thus, $\alpha_r p_{rs} = p_{rk}=\alpha_r$, so $p_{rs}=1$, a contradiction.

(3) Suppose $\eta_{r\ell},\eta_{rk} \neq 0$ with $\ell < k$.
Again by (1), $\ell,k \neq r$. By (2), we have $x \cdot u_\ell = \eta_{r\ell} u_r$ and $x \cdot u_k = \eta_{rk} u_r$.
A computation as above shows that the $u_r^2$ coefficient in $x \cdot (u_ru_\ell-p_{r\ell}u_\ell u_r)$ is $\eta_{r\ell}(\alpha_r-p_{r\ell})$ and in $x \cdot (u_ru_k-p_{rk}u_ku_r)$ is $\eta_{rk} (\alpha_r-p_{rk})$. Now
\begin{align*}
0	&= x \cdot (u_\ell u_k - p_{\ell k} u_k u_\ell) \\
	&= \left( \alpha_\ell u_\ell \left(\eta_{rk} u_r\right) + \left(\eta_{r\ell} u_r\right)u_k\right)
		- p_{\ell k} \left( \alpha_k u_k \left(\eta_{r\ell} u_r\right) + \left(\eta_{rk} u_r\right)u_\ell\right) \\
	&= \eta_{rk} (\alpha_\ell - p_{\ell k} p_{r \ell}) u_\ell u_r
		+ \eta_{r\ell} (p_{rk} - p_{\ell k} \alpha_k) u_k u_r.
\end{align*}
By \eqref{eq.coeff}, $\alpha_r = \lambda \alpha_\ell=\lambda \alpha_k$, so $\alpha_\ell = \alpha_k$. Then
\begin{align*}
\alpha_\ell = p_{\ell k} p_{r \ell} = \alpha_k\inv p_{rk} p_{r\ell} = \alpha_k\inv \alpha_r^2 = \alpha_\ell \lambda^2.
\end{align*}
Thus $\lambda^2=1$, a contradiction.

(4) Assume $i,j,k,\ell$ are all distinct and $\eta_{ij},\eta_{k\ell} \neq 0$, so necessarily $t>3$.
Also, by (2) and (3), these are the distinct nonzero elements in their respective row and column.
The coefficient of $u_k u_i$ in $x \cdot (u_i u_\ell - p_{i \ell} u_\ell u_i)$ is $\eta_{k\ell}(\alpha_i p_{ik}-p_{i\ell})$ and in $x \cdot (u_j u_k - p_{jk} u_k u_j)$ it is $\eta_{ij}(p_{ik}-\alpha_k p_{jk})$. Moreover,
\[ x \cdot (u_ju_\ell - p_{j\ell} u_\ell u_j)
	= \eta_{ij}(p_{i\ell} - \alpha_\ell p_{j\ell}) u_\ell u_i - \eta_{k\ell}(\alpha_jp_{jk}-p_{j\ell} )u_k u_j.\]
Because $\eta_{ij},\eta_{k\ell} \neq 0$, then by \eqref{eq.coeff}, $\alpha_i=\lambda \alpha_j$ and $\alpha_k = \lambda \alpha_\ell$. Hence,
\[
\alpha_\ell\inv \alpha_i p_{ik} = \alpha_\ell\inv p_{i\ell} = p_{j\ell}
= \alpha_j p_{jk} = \alpha_j\alpha_k\inv p_{ik} = \lambda^{-2} \alpha_i\alpha_\ell\inv p_{ik}.
\]
It follows that $\lambda^2=1$, a contradiction.

(5) When $t=2$, the matrix $x$ is nilpotent by (1). 
Let $\eta_{ij}$ be a nonzero entry in $x$.
After possibly renumbering generators, we may assume that $j>i$.
By (2) and (3), $\eta_{ij}$ is the only nonzero entry in its row and column.
Moreover, by (4), the only other possible nonzero entries 
are of the form $\eta_{\ell i}$ or $\eta_{jk}$ for some $\ell \neq i$ and $k \neq j$.
Suppose both are nonzero.
By (1), we also have $\ell \neq j$ and $k \neq i$.
If $\ell \neq k$, then $\eta_{\ell i}\eta_{jk}=0$ by (4), so $\ell =k$.
But then by \eqref{eq.coeff}, we have
$\alpha_k = \lambda \alpha_i = \lambda^2 \alpha_ j = \lambda^3 \alpha_k$, so $\lambda^3=1$, a contradiction.
Thus, at most one of $\eta_{\ell i}$ or $\eta_{jk}$ is nonzero and it is clear that $x$ is nilpotent.

(6) If $\ord(g)=m$, then the result is clear. Assume $\ord(g)\neq m$ and $t \leq m \neq 3$. By (5), $x$ is nilpotent so $g$ acts diagonally by Lemma \ref{lem.diagonal}.
Since $x$ acts linearly, then $x^k=0$ for some $k \leq t$, so $0=x^m=\gamma(g^m-1)$. Thus, either $\gamma=0$ or $g^m$ acts trivially.
In the latter case, $\ord(g)=m$ by the inner faithful hypothesis, but this contradicts our hypotheses, so $\gamma=0$.
\end{proof}

Without the hypothesis that $m \neq 3$, it is possible to have actions of rank one pointed Hopf algebras in which $x$ is not nilpotent.

\begin{example}
\label{ex.nnilp}
Suppose $m=3$ and, for simplicity, assume $t=3$.
Let $\lambda$ be a primitive third root of unity.
We will consider a generalized Taft algebra action on $\kk_\bp[u_1,u_2,u_3]$.
Let $g=\diag(\alpha_1,\alpha_2,\alpha_3)$ and assume the nonzero entries in $x$
are $\eta_{12},\eta_{23},\eta_{31}$.
Observe that $x$ is not nilpotent.
We have
\begin{align*}
x \cdot (u_1u_2-p_{12}u_2u_1) 
    &= \eta_{12}(\alpha_1-p_{12})u_1^2 + \eta_{31}(1-\alpha_2p_{12}p_{23})u_3u_2, \\
x \cdot (u_2u_3-p_{23}u_3u_2) 
    &= \eta_{23}(\alpha_2-p_{23})u_2^2 + \eta_{12}(1-\alpha_3p_{23}p_{31})u_3u_1, \\
x \cdot (u_3u_1-p_{31}u_1u_3) 
    &= \eta_{31}(\alpha_3-p_{31})u_3^2 + \eta_{23}(1-\alpha_1p_{31}p_{12})u_2u_1.  
\end{align*}
Let $\alpha$ be a primitive ninth root of unity such that $\lambda=\alpha^{-3}$.
Set $\alpha_3=\alpha$, $\alpha_2=\lambda\alpha$, and $\alpha_1=\lambda^2\alpha$.
Hence, by \eqref{eq.coeff}, $(gx-\lambda xg) \cdot u_i=0$ for all $i$.
Set $p_{12}=\lambda^2\alpha$, $p_{23}=\lambda\alpha$ and $p_{31}=\alpha$.
It now follows that
\[ \alpha_2p_{12}p_{23} = (\lambda\alpha)(\lambda^2\alpha)(\lambda\alpha) = \lambda\alpha^3=1,\]
and so the first equation above vanishes.
One verifies similarly that the remaining equations vanish.
Now we see that $(g^3-1) \cdot u_i = (\alpha^3-1)u_i$ and $x^3 \cdot u_i=u_i$.
Since $\alpha^3 = \lambda\inv \neq 1$, then we set $\gamma=(\alpha^3-1)\inv \neq 0$ and so the above defines an action of $T_n(\lambda,3,\gamma)$ on $\kk_\bp[u_1,u_2,u_3]$.
\end{example}

Proposition \ref{prop.genaction} and Lemma \ref{lem.xprops} give some insight into the actions of rank 1 pointed Hopf algebras on $A=\qp$ or $\wa$.
Let $H=H(G,g,\chi,\gamma)$ and assume $x$ is nilpotent (for example, when the hypotheses of Lemma \ref{lem.xprops} (6) are satisfied).
We know by \cite{KR} that $G$ must be finite and because the elements of $G$ act diagonally as automorphisms on $A$, we have that $G$ is abelian.
The distinguished element $g \in G$ acts according to Proposition \ref{prop.genaction}.
Let $a \in G$ and assume that $\eta_{12}\neq 0$ in $x$.
Then $a$ and $x$ satisfy \eqref{eq.coeff} but for the corresponding $\chi(a)$ in place of $\lambda$.
Then, when $a$ is considered as an element of $\Aut(A)$ it takes the form
\[ a=\begin{pmatrix}\chi(a)\beta & 0 \\ 0 & \beta\end{pmatrix},	\quad\beta \in \kk^\times.\]

We now restrict our study to the subalgebra of $H(G,g,\chi,0)$ that is generated by $g$ and $x$.
Recall that this is a generalized Taft algebra.
We remark that, by Proposition \ref{prop.genaction}, an action of Type (a) on $\qp$ is the same as an action of Type (b) on $\kk_{\mu\inv}[v,u]$.
Thus, we will henceforth assume that all actions on a quantum plane
are of Type (a) but will differentiate between the two algebras even though they are isomorphic.

We will assume throughout that all parameters are roots of unity of order at least three. The reason for this restriction is to avoid getting bogged down in special cases. We will show that there are only two types of actions. The first is just trivial extensions of actions on quantum planes. In certain cases, there are trivial extensions of actions on quantum 3-spaces as described below.

\begin{example}
\label{ex.qas}
This is a generalization of \cite[Example 2.1]{GWY}.
Let $A=\kk_\bp[u_1,u_2,u_3]$ such that $\ord(p_{ij})>2$ for all $i \neq j$.
Define a linear action of $T_n(\lambda,m,0)$, $m>2$, on $A$ such that $g$
acts diagonally and the only nonzero elements of $x$ are $\eta_{12},\eta_{23}$.
By \eqref{eq.coeff}, $\alpha_1=\lambda \alpha_2=\lambda^2 \alpha_3$. 
We borrow computations from Example \ref{ex.nnilp},
but here $\eta_{31}=0$.
Hence, $\alpha_1=p_{12}$ and $\alpha_2=p_{23}$. Furthermore, 
\[ \alpha_1^2p_{31} = \alpha_1p_{31}p_{12} = 1 = \alpha_3p_{23}p_{31} = (\lambda^{-2}\alpha_1)(\lambda^{-1}\alpha_1)p_{31}.\]
This implies that $\lambda^3=1$ and that $p_{13}=\alpha_1^2$.
\end{example}

Given a quantum affine space $A=\kk_\bp[u_1,\hdots,u_t]$,
we say $T=T_n(\lambda,m,0)$ acts as a trivial extension of an action on the quantum affine 3-space subalgebra $A_{ijk}$ if the action is as given in Example \ref{ex.qas}.
That is, $g$ acts diagonally on $A$, $x \cdot u_j=\eta_{ij}u_i$, $x \cdot u_k=\eta_{jk}u_j$, and $x \cdot u_\ell=0$ for all $\ell \neq j,k$.

We remark briefly that Example \ref{ex.qas} does not extend beyond the $t=3$ case. For example, suppose $T_n(\lambda,m,0)$, $m>2$, acts on $A=\kk_\bp[u_1,u_2,u_3,u_4]$, $\ord(p_{ij})>2$ for all $i \neq j$.
If $g$ acts diagonally and $x$ acts by
\[ x \cdot u_1=0, \quad x \cdot u_2=\eta_{12}u_1, \quad 
x \cdot u_3=\eta_{23}u_2, \quad  x \cdot u_4=\eta_{34}u_4,\]
then $\eta_{12}\eta_{34}=0$ by Lemma \ref{lem.xprops} (4).

\begin{theorem}
\label{thm.nspace}
Suppose $T_n(\lambda,m,0)$ acts linearly and inner faithfully on $A=\kk_\bp[u_1,u_2,\cdots,u_t]$.
Assume $m$, $t$, and $\ord(p_{ij})$ for $i \neq j$ are all at least $3$.
Then every action is a trivial extension of an action on some $A_{ij}$ or $A_{ijk}$.
\end{theorem}

\begin{proof}
By Lemma \ref{lem.diagonal}, $g$ acts diagonally on $A$.
If $x \neq 0$, then after a change of variable we may assume that $\eta_{12}\neq 0$. By Lemma \ref{lem.xprops}, this implies that the only other possible nonzero entries may be $\eta_{23}$ and $\eta_{31}$.
If they are all nonzero then we are in the situation of Example \ref{ex.nnilp}, whence $x$ is not nilpotent, a contradiction.
On the other hand, if $\eta_{23}=\eta_{31}=0$, then the action is a trivial extension of an action on a quantum plane.
Finally, if exactly one of $\eta_{23}$ or $\eta_{31}$ is nonzero, then we are in the setting of Example \ref{ex.qas}.
\end{proof}

Next we aim to understand actions of $B(G,\gu,\cu)$ on quantum affine spaces. Our primary goal will be to determine the maximum rank of such a $B$ and we do this by determining how to ``patch'' together actions of generalized Taft algebras.

By Lemma \ref{lem.diagonal}, we may assume that all of the $g_i$ act diagonally.
In light of Theorem \ref{thm.nspace}, we may assume that $x_1=(\eta_{ij})$ has nonzero entry $\eta_{12}$ and at most one other nonzero entry, either $\eta_{23}$ or $\eta_{31}$. 
After a change of variable, we may assume in either of the latter cases that $\eta_{12},\eta_{23} \neq 0$.

We begin by considering the above question for actions on quantum planes and quantum Weyl algebras.

\begin{lemma}
\label{lem.Baction}
Suppose $B(G,\gu,\cu)$ has rank $\theta$, and that $B$ acts linearly and inner faithfully on $A=\qp$ or $\wa$.
Assume $\ord(\mu)$ and $m_i$ for all $i$ are at least 3. Then
\begin{enumerate}
\item $\ord(\mu) \mid n_i$ and $B_i$ acts on $A$ according to Proposition \ref{prop.genaction} for each $i=1,\hdots,\theta$;
\item either all $B_i$ act according to Proposition \ref{prop.genaction} (a) or all act according to Proposition \ref{prop.genaction} (b);
\item for all $i \neq j$, we have $\lambda_i = \chi_j(g_i)$.
\end{enumerate}
\end{lemma}

\begin{proof}
Suppose $B$ acts linearly on $A$ such that $x_i \cdot A\neq 0$.
Since $B_i \iso T_{n_i}(\lambda_i,m_i,0)$ as Hopf algebras,
then $B_i$ acts linearly on $A$ and $x_i \cdot A \neq 0$.
Thus, the conditions in Proposition \ref{prop.genaction} are necessary and the action is the one given in that result. It follows that $\ord(\mu) \mid  n_i$ for each $i$.
We will show that all the $B_i$ act according to (a) or (b).

Without loss of generality, suppose $x_1$ acts on $A$ according to (a) and $x_2$ acts according to (b), then
\[ (x_1x_2 - \chi_2(g_1)x_2x_1) \cdot v = 0 - \chi_2(g_1)( x_2 \cdot (\eta_1 u)) = -\eta_1\eta_2\chi_2(g_1)v \neq 0,\]
a contradiction. Hence, after a linear change of variable we may assume that each $B_i$ acts according to (a). For $j=1,2$, we write $x_j \cdot v=\eta_j u$, $\eta_j \in \kk^\times$. If $i \neq j$, then
\begin{align*}
(g_ix_j - \chi_j(g_i)x_jg_i) \cdot v 
	&= g_i \cdot (\eta_j u) - \chi_j(g_i)(x_j \cdot (\lambda_i\inv \mu v))
	= \eta_j\mu ( 1 - \chi_j(g_i)\lambda_i\inv) u.
\end{align*}
Thus, $\lambda_i = \chi_j(g_i)$.
\end{proof}

We now proceed to study quantum affine spaces in general.

\begin{lemma}
\label{lem.Brank}
Suppose $B(G,\gu,\cu)$ has rank $\theta$, and that $B$ acts linearly and inner faithfully on $A=\qp$ or $\wa$.
Assume $\ord(\mu)$ and $m_i$ for all $i$ are at least 3.
Then $\theta \leq 2$, and if $\theta=2$, then there exists a primitive $m^{th}$ root of unity $\omega$ such that
$\chi_1(g_1) = \chi_2(g_1) = \omega$ and 
$\chi_1(g_2) = \chi_2(g_2) = \omega\inv$.
\end{lemma}

\begin{proof}
The case $\theta=1$ is handled by Proposition \ref{prop.genaction}. 
Suppose $\theta=2$ and set $\lambda_1 = \chi_1(g_1) = \omega$. 
By Lemma \ref{lem.Baction} (3), $\chi_2(g_1) = \lambda_1 = \omega$.
Moreover, the relations of $B$ imply that $\chi_1(g_2)=\chi_2(g_1)\inv = \omega\inv$.
Applying Lemma \ref{lem.Baction} (3) again, we have $\lambda_2 = \chi_2(g_2) = \chi_2(g_1)\inv = \omega\inv$.

Now suppose $\theta \geq 3$. Using the same logic as above we have
$\chi_2(g_3)=\chi_3(g_2)\inv = \omega$. But then
\[ \omega = \chi_2(g_3) = \chi_1(g_3) = \chi_3(g_1)\inv = \chi_1(g_1)\inv = \omega\inv.\]
Thus, $2 \geq \ord(\omega) = \ord(\lambda_1)$, contradicting our hypothesis.
\end{proof}

\begin{lemma}
\label{lem.genpatch}
Suppose $B(G,\gu,\cu)$ has rank $\theta\geq 2$, and that $B$ acts linearly and inner faithfully on $A=\kk_\bp[u_1,\hdots,u_t]$.
Assume $t$, $m_i$ for all $i$, and $\ord(p_{ij})$ for all $i \neq j$ are at least 3. Write $x_1=(\eta_{ij})$ and $x_2=(\mu_{ij})$.
\begin{enumerate}

\item If $\eta_{ij},\mu_{jk} \neq 0$, then $k\neq i$ and both $B_1$ and $B_2$ act as trivial extensions of an action on $A_{ijk}$.
\item If $\eta_{ij},\mu_{kj} \neq 0$, then $\lambda_1=\lambda_2\inv$.
\item There may be at most two $x_i$ with nonzero entries in the same column.
\end{enumerate}
\end{lemma}

\begin{proof}
(1) Assume $\eta_{ij},\mu_{jk} \neq 0$.

First, suppose $k=i$. 
If $x_1 \cdot u_i = 0$, then $0 = (x_1 x_2 - \chi_2(g_1) x_2 x_1) \cdot u_i = \mu_{ji} \eta_{ij} u_i$, a contradiction. 
By Lemma~\ref{lem.xprops}~(1, 2), there is some $\ell \neq i,j$ such that $x_1 \cdot u_i = \eta_{\ell i} u_\ell$. 
Similarly, we must have $x_2 \cdot u_\ell \neq 0$, so by Lemma~\ref{lem.xprops}~(1, 3), there is some $m$ such that $x_2 \cdot u_\ell = \mu_{m \ell} u_m$.
From this, and through similar computations for the second, we have
\begin{align*}
 0 &= (x_1 x_2 - \chi_2(g_1) x_2 x_1) \cdot u_i
   = \mu_{ji}\eta_{ij} u_i -\chi_2(g_1)\eta_{\ell i}\mu_{m \ell} u_m, \\
 0 &= (x_1x_2-\chi_2(g_1)x_2x_1) \cdot u_j
   = \mu_{\ell' j}\eta_{m' \ell'} u_{m'} - \chi_2(g_1)\eta_{ij}\mu_{ji} u_j.
\end{align*}
By the assumption that $\mu_{ji} \eta_{ij} \neq 0$, we must have that $m = i$ and $m' = j$.
Since $\eta_{\ell i}, \eta_{j \ell'} \neq 0$, we cannot have $i,j,\ell, \ell'$ all distinct by Lemma~\ref{lem.xprops}~(4).
This forces $\ell=\ell'$, but then $x_1$ is not nilpotent.
We conclude that $k \neq i$.

In general, for $k \neq i$, the same argument shows that there exists $\ell \neq i,k$ such that $\eta_{\ell k},\mu_{i\ell} \neq 0$ and by Lemma~\ref{lem.xprops}~(4), we cannot have $i,j,k,\ell$ all distinct, so $\ell=j$.

(2) Assume $\eta_{ij},\mu_{kj} \neq 0$. We write $g_1 \cdot u_i = \alpha_i u_i$ and $g_2 \cdot u_i = \beta_i u_i$ for all $i$. If $k=i$, then 
\begin{align*}
0 &= (g_2x_1-\chi_1(g_2)x_1g_2) \cdot u_j =\eta_{ij} (\beta_i - \chi_1(g_2) \beta_j) u_i, \\
0 &= (g_1x_2-\chi_2(g_1)x_2g_1) \cdot u_j =\mu_{ij} (\alpha_i - \chi_2(g_1) \alpha_j) u_i.
\end{align*}
By \eqref{eq.coeff}, $\alpha_i = \lambda_1 \alpha_j$ and $\beta_i = \lambda_2 \beta_j$. Hence, $\lambda_1 = \chi_2(g_1)$ and $\lambda_2 = \chi_1(g_2)$, and the result follows.

Next suppose that $k \neq i$.
Then the coefficient of $u_ku_i$ in $x_1 \cdot (u_ju_k-p_{jk}u_ku_j)$ is $\eta_{ij}(p_{ik}-p_{jk}\alpha_k)$ and in $x_2 \cdot (u_iu_j-p_{ij}u_ju_i)$ it is $\mu_{kj}(\beta_ip_{ik}-p_{ij})$.
Thus, $\alpha_k = p_{ik}p_{kj}$ and $\beta_i=p_{ij}p_{ki}$. Now
\begin{align*}
(g_2x_1-\chi_1(g_2)x_1g_2) \cdot u_j 
    &= \eta_{ij}(\beta_i-\chi_1(g_2)\beta_j)u_i, \\
(g_1x_2-\chi_2(g_1)x_2g_1) \cdot u_j
    &= \mu_{jk}(\alpha_k-\chi_2(g_1)\alpha_j)u_k.
\end{align*}
By Proposition \ref{prop.genaction}, $\alpha_j = \lambda_1\inv p_{ij}$ and $\beta_j = \lambda_2\inv p_{kj}$, so $\chi_1(g_2)=\lambda_2\beta_i p_{jk}$ and $\chi_2(g_1)=\lambda_1\alpha_k p_{ji}$. Now
\[ 1 = \chi_2(g_1)\chi_1(g_2) 
	= \lambda_2\beta_i p_{jk}\lambda_1\alpha_k p_{ji}
	= \lambda_1\lambda_2 (p_{ij}p_{ki})(p_{ik}p_{kj})(p_{jk}p_{ji}) 
	= \lambda_1\lambda_2,\]
as claimed.

(3) Suppose three $x_i$, say $i=1,2,3$, have nonzero entries in the same column. Then by (2), we would have 
$\lambda_1=\lambda_2\inv$, $\lambda_1=\lambda_3\inv$, and $\lambda_2=\lambda_3\inv$,
whence $\lambda_3^2=1$, a contradiction.
\end{proof}

The following result is proved for $t=2$ in Lemma \ref{lem.Brank}.

\begin{theorem}
\label{thm.QLSrank}
Suppose $B=B(G,\gu,\cu)$ has rank $\theta$, and that $B$ acts linearly and inner faithfully on $A=\kk_\bp[u_1,\hdots,u_t]$, $t \geq 2$.
Assume $m_i$ for all $i$ and $\ord(p_{ij})$ for all $i \neq j$ are at least 3. 
Then $\theta \leq 2(t-1)$.
\end{theorem}

\begin{proof}
Let $\Gamma$ be a directed graph with $t$ vertices $v_1,\hdots,v_t$ corresponding to the generators of $A$. We draw an arrow from $v_j$ to $v_i$ if the $(i,j)$ entry of some $x_k$ is nonzero.
Let $\Gamma_1$ denote the number of arrows in $\Gamma$.
It is clear that $\theta \leq \Gamma_1$.

By Lemma \ref{lem.genpatch} (3), a vertex may not be the source of more than two arrows, and so $\Gamma_1 \leq 2t$.
Lemma \ref{lem.genpatch} (1) implies that $\Gamma$ contains no two cycles, and if there is a path of length two, then some $B_\ell$ acts as a trivial extension of some $A_{ijk}$. That is, two arrows correspond to the same action. 
Now if $\Gamma_1 \leq 2t-2$, we are done, since $\theta \leq \Gamma_1$.
If $\Gamma_1 = 2t-1$, then the target of any arrow is the source of at least one other, giving a path of length 2.
Hence, $\theta \leq \Gamma_1 - 1 = 2t - 2$.
If $\Gamma_1 = 2t$, then the target of any arrow is the source of two others, giving two paths of length 2. 
Hence, $\theta \leq \Gamma_1 - 2 = 2t - 2$.
\end{proof}

Lemma \ref{lem.Brank} shows that the bound in Theorem \ref{thm.QLSrank} is sharp when $t=2$. The next example shows this for $t>2$.

\begin{example}
Let $A=\kk_\bp[u_1,\hdots,u_t]$, with $t$ and $\ord(p_{ij})$ for $i \neq j$ all at least 3. We will construct an action of some $B(G,\gu,\cu)$ of rank $2(t-1)$ on $A$.

First, let $B_k \iso T_{n_k}(\lambda_k,m_k)$, $k=1,\hdots,t-1$, have canonical generators $\{g_k,x_k\}$. We will assume $B_k$ acts as a trivial extension of an action on $A_{1(k+1)}$ with $g_k \cdot u_i = \alpha_{ki} u_i$ for some $\alpha_{ki} \in \kk^\times$ and $x_k \cdot u_{k+1}=u_1$. By definition, $x_k \cdot u_i=0$ for all $i \neq k+1$.
By Proposition \ref{prop.genaction} we must have $\alpha_{i1}=p_{1(i+1)}$.
Furthermore, for $i \neq j$,
\begin{align}
\label{eq.strict1} x_i \cdot (u_{i+1}u_{j+1}-p_{(i+1)(j+1)}u_{j+1}u_{i+1}) &= (p_{1(j+1)}-p_{(i+1)(j+1)}\alpha_{i(j+1)})u_{j+1} u_1, \\
\notag  (g_ix_j-\chi_j(g_i)x_jg_i) \cdot u_{j+1} &=(\alpha_{i1}-\chi_j(g_i)\alpha_{i(j+1)})u_1. 
\end{align}
Hence, we have
\[ \chi_j(g_i)\chi_i(g_j)
    = (\alpha_{i1}\alpha_{i(j+1)}\inv)(\alpha_{j1}\alpha_{j(i+1)}\inv)
    = p_{1(i+1)}(p_{1(j+1)}p_{(i+1)(j+1)}\inv)\inv
    p_{1(j+1)}(p_{1(i+1)}p_{(i+1)(j+1)})\inv = 1.\]
It follows that all compatibility conditions are met amongst the $\{g_k,x_k\}$.

In a similar way, set $B_{k-t+1}' = B_k \iso T_{n_k}(\lambda_k\inv,m_k)$, $k=t,\hdots,2t-2$ and denote the canonical generators by $\{g_k',x_k'\}$. We will assume $B_k'$ acts as a trivial extension of an action on $A_{1(k+1)}$ with $g_k' \cdot u_i = \beta_{ki} u_i$ for some $\beta_{ki} \in \kk^\times$ and $x_k' \cdot u_{k+1}=u_1$ with $x_k' \cdot u_i=0$ for all $i \neq k+1$.
The argument above shows that the compatibility conditions amongst the $B_k'$ are met. It remains to show that the $B_k$ and the $B_k'$ are pairwise compatible. We have
\begin{align*}
(g_ix_j'-\chi_j'(g_i)x_j'g_i) \cdot u_{j+1}
    &= (\alpha_{i1}-\chi_j'(g_i)\alpha_{i(j+1)})u_1 \\
(g_j'x_i-\chi_i(g_j')x_ig_j') \cdot u_{i+1}
    &= (\beta_{j1}-\chi_i(g_j')\beta_{j(i+1)})u_1.
\end{align*}
A computation as in \eqref{eq.strict1} shows that $\beta_{ij}=\alpha_{ij}$ for all $i,j$. Thus, we have
\[ \chi_j'(g_i)\chi_i(g_j') = \alpha_{i1}\alpha_{i(j+1)}\inv\beta_{j1}\beta_{j(i+1)}^{-1}
    = \alpha_{i1}\alpha_{i(j+1)}\inv\alpha_{j1}\alpha_{j(i+1)}^{-1} 
    = 1.
\]
\end{example}

\section{Quantum Matrix Algebras}

  We want to classify actions of $T_n(\lambda, m, 0)$ on $\mc O_q(M_2(\kk))$ with $x$ acting linearly and nonzero.
  To do this, we first note some automorphisms of $\mc O_q(M_N(\kk))$.
  First, let $\mc H$ denote the group $(\kk^\times)^{2N-1}$. 
  Each element $(a_1, \ldots, a_N, b_1, \ldots, b_{N-1})$ of $\mc H$ gives a unique automorphism of $\mc O_q(M_N(\kk))$ by 
  \[
    Y_{ij} \mapsto 
      \begin{cases} 
        a_i b_j Y_{ij}, & (j<N) \\ 
        a_i Y_{ij}, & (j=N). 
      \end{cases}
  \]
  Equivalently, if $\alpha_{ij} \in \kk$ is defined by $g \cdot Y_{ij} = \alpha_{ij} Y_{ij}$, then the matrix $(\alpha_{ij})$ forms an $N \times N$ matrix of rank 1 with no zero entries.
  For example, if $N=2$ and $g \in \mc H$, then
    \[
      g \cdot A = \alpha_{11} A \quad \quad
        g \cdot B = \alpha_{12} B \quad \quad
        g \cdot C = \alpha_{21} C \quad \quad
        g \cdot D = \alpha_{22} D,
    \]
    for $\alpha_{ij} \in \kk^\times$.
    
  The transposition map, $\tau$, given by $Y_{ij} \mapsto Y_{ji}$, gives another automorphism of $\mc O_q(M_N(\kk))$.
  It was conjectured in \cite{ll-aoqm} that if $q$ is not a root of unity, then $\Aut(\mc O_q(M_N(\kk))) \cong \mc H \rtimes \langle \tau \rangle$, and this conjecture was proven in its entirety in \cite{y-llc}.
  
  \begin{remark} \label{rem:iso}
    Note that if $h \cdot a$ denotes an action of a Hopf algebra $H$ on an associative $\kk$-algebra $A$, and $\phi: A \to B$ is a $\kk$-algebra isomorphism, then $h \circ b \defeq \phi ( h \cdot \phi^{-1}(b))$ gives an action of $H$ on $B$.
    We will use this fact frequently, with either $\phi = \tau$ or $\phi: \mc O_q(M_N(\kk)) \to \mc O_{q^{-1}}(M_N(\kk))$ being the map $Y_{ij} \mapsto Y_{(N+1-i)(N+1-j)}$.
  \end{remark}


\begin{lemma} \label{lem:2matrixdiag}
  Let $q \neq \pm 1$ and assume $m \geq 3$.
  If $T_n(\lambda, m, 0)$ acts on $\mc O_q(M_2(\kk))$ with $x$ acting linearly and nonzero, and $g$ acting as an element of $\mc H \rtimes \langle \tau \rangle$, then $g$ must act as an element of $\mc H$, i.e. diagonally on the basis $(A, B, C, D)$ of $\mc O_q(M_2(\kk))_{(1)}$.
\end{lemma}

  \begin{proof}
      Let $(\eta_{ij})$ give the action of $x$ on the generators, i.e. $(\eta_{ij})$ is the matrix representing the action of $x$ on the basis $(A,B,C,D)$ of the $1$-graded piece.
    
    Suppose $g$ does not act as an element of $\mc H$, so by the assumption that $g \in \mc H \rtimes \langle \tau \rangle$,
    \[
      g \cdot A = \alpha_{11} A \quad \quad
        g \cdot B = \alpha_{21} C \quad \quad
        g \cdot C = \alpha_{12} B \quad \quad
        g \cdot D = \alpha_{22} D,
    \]
    for $(\alpha_{ij})_{i,j}$ a rank one matrix with no zero entries.
    On the ordered basis $(A,B,C,D)$, we have that $gx - \lambda xg$ is given by
    \[
      \left(
        \begin{array}{cccc}
          \eta_{11} \alpha_{11} (1 - \lambda) 
            & (\eta_{12} \alpha_{11} - \lambda \eta_{13} \alpha_{21})
            & \eta_{13} \alpha_{11}  - \lambda \eta_{12} \alpha_{12}
            & \eta_{14} (\alpha_{11} -\lambda \alpha_{22}) \\
          \eta_{31} \alpha_{12} - \lambda \eta_{21} \alpha_{11}
            & \eta_{32} \alpha_{12} - \lambda \eta_{23} \alpha_{21}
            & \alpha_{12} (\eta_{33} - \lambda \eta_{22}) 
            & \eta_{34} \alpha_{12} - \lambda \eta_{24} \alpha_{22} \\
          \eta_{21} \alpha_{21} - \lambda \eta_{31} \alpha_{11}
            & \alpha_{21} (\eta_{22} - \lambda \eta_{33})
            & \eta_{23} \alpha_{21} - \lambda \eta_{32} \alpha_{12}
            & \eta_{24}  \alpha_{21} - \lambda \eta_{34} \alpha_{22} \\
          \eta_{41} (\alpha_{22} - \lambda \alpha_{11})
            & \eta_{42} \alpha_{22} - \lambda \eta_{43} \alpha_{21}
            & \eta_{43} \alpha_{22} - \lambda \eta_{42} \alpha_{12} 
            & \eta_{44} \alpha_{22} (1 - \lambda)
        \end{array}
      \right).
    \]
    Since this must be the zero matrix, we have that $\eta_{11} = \eta_{44} = 0$.
    Moreover, the following pairs are either both zero or both nonzero:
    \[
      (\eta_{12},\eta_{13}), \quad (\eta_{21},\eta_{31}), \quad (\eta_{22},\eta_{33}), \quad (\eta_{23},\eta_{32}), \quad (\eta_{24},\eta_{34}), \quad (\eta_{42},\eta_{43}).
    \]
    If $\eta_{23} \neq 0$, then $\eta_{32} = \frac {\lambda \alpha_{21}}{\alpha_{12}} \eta_{23}$ and $\eta_{23} = \frac {\lambda \alpha_{12}}{\alpha_{21}} \eta_{32}$.
    Therefore, $\lambda^2 = 1$, a contradiction to $m \neq 2$.
    Hence, $\eta_{23} = \eta_{32} = 0$.
    Similarly, we have $\eta_{22} = \eta_{33} = 0$.
    Now note that $(A^2, AB, AC, AD, B^2, BC, BD, C^2, CD, D^2)$ is a basis for $\mc O_q(M_2(\kk))_2$.
    The $B^2$ coefficient of $x \cdot (AB - qBA)$ is $\eta_{21}$.
    Thus, $\eta_{21} = \eta_{31} = 0$.
    The $BD$ coefficient is $\eta_{41} q^{-1}$, so $\eta_{41} = 0$ as well, giving $x \cdot A = 0$.
    The $BC$ coefficient is now $(q^2 - 1) \eta_{42}$.
    Since $q^2 \neq 1$, we have $\eta_{42} = \eta_{43} = 0$.
    The coefficients of $AB, \ B^2,$ and $BC$ in $x \cdot (BD - qDB)$ are, respectively, $-q \eta_{14}$, $-q \eta_{24}$, and $\alpha_{12} \eta_{24} + (q^2 - 1) \alpha_{22} \eta_{12} - q \eta_{34}$.
    Thus, as above, we get that $\eta_{14} = \eta_{24} = \eta_{34} = \eta_{12} = \eta_{13} = 0$.
    Therefore, $x$ acts by zero, a contradiction.
  \end{proof}
  
  Before classifying linear actions of $T_n(\lambda, m, 0)$ on $\mc O_q(M_2(\kk))$ in general, we consider a special case.
  
   \begin{example} \label{ex:ordqis3}
   \label{ex.qorder3}
  Let $q \in \kk$ with $\ord(q)=3$. Also assume that $m \geq 3$.
  The following give actions of $T_n(\lambda, m, 0)$ on $\mc O_q(M_2(\kk))$. 
  (The action of $g$ is specified as an element of $\mc H$, i.e. a matrix of rank one $(\alpha_{ij})_{i,j}$ so that $g \cdot Y_{ij} = \alpha_{ij} Y_{ij}$.)
  \begin{align*}
    \text{(1)} & \quad g = \begin{pmatrix}1 & q^{-1} \\ q^{-1} & q \end{pmatrix}, \quad 
      x \cdot D = \gamma A, \quad 
      x \cdot A = \delta B + \epsilon C, \quad 
      (\gamma, \delta, \epsilon \in \kk ; \ \lambda = q^2) \\
    \text{(2)} & \quad g = \begin{pmatrix}q^{-1}  & q   \\ q      & 1\end{pmatrix}, \quad
      x \cdot A = \gamma D, \quad
      x \cdot D = \delta B + \epsilon C, \quad
      (\gamma, \delta, \epsilon \in \kk ; \ \lambda = q^{-2})
  \end{align*}
  \end{example}
  
  \begin{proposition} \label{prop:matrixactions}
    Let $q \neq \pm 1$ and assume $m \geq 3$.
    Then $T_n(\lambda, m, 0)$ acts on $\mc O_q(M_2(\kk))$ with $x$ acting linearly and nonzero, and $g$ acting as an element of $\mc H \rtimes \langle \tau \rangle$, if and only if     
    \begin{itemize}
    \item $\lambda = q^{\pm 2}$ and $\ord(q) | n$, or
    \item $\lambda = q^{\pm 4}$ and $\ord(q^2) | n$.
    \end{itemize}
    The actions are given as in Example~\ref{ex:ordqis3} (with $\gamma, \delta, \epsilon$ not all zero) and Table~\ref{table.mat2}.
  \end{proposition}
    
  Noting that $\mc O_q(M_2(\kk)) \cong \mc O_{q^{-1}}(M_2(\kk))$ via $A \mapsto D, \ B \mapsto C, \ C \mapsto B, \ D \mapsto A$, we see that by changing $q$, we can assume that $\lambda = q^2$ or $\lambda = q^4$, with the action coming from the corresponding list.
  
  \begin{table}[h]
\caption[mat2]{Actions of $T_n(\lambda, m, 0)$ on $\mc O_q(M_2(\kk))$. 
The second column lists the value of $\lambda$ in terms of $q$ while the third indicates the action of $g$ as an element of $\mc H$, i.e. a matrix of rank one $(\alpha_{ij})_{i,j}$ so that $g \cdot Y_{ij} = \alpha_{ij} Y_{ij}$.
The fourth column indicates how $x$ acts on the generators with $\delta, \epsilon \in \kk$.
We assume the action is trivial if not listed and the last column lists any restrictions on $\delta$ and $\epsilon$.}\label{table.mat2}
\begin{tabular}{|c|c|c|c|c|}
\hline
 & $\lambda$	& action of $g$	& action of $x$	&	restrictions on $\delta,\epsilon$ \\ \hline	

1 & $q^2$ 	& $\begin{pmatrix}q  &  q^{-1} \\ q  &  q^{-1}\end{pmatrix}$	&   
$x \cdot B = \delta A, \ x \cdot D = \delta C$	& $\delta\neq 0$ \\ \hline

2 & $q^2$	& $\begin{pmatrix}q      & q      \\ q^{-1} & q^{-1}\end{pmatrix}$ &
$x \cdot C = \delta A, \ x \cdot D = \delta B$ & $\delta\neq 0$ \\ \hline

3 & $q^2$	& $\begin{pmatrix}q^{-3} & q^{-1}  \\ q^{-1} & q\end{pmatrix}$ &
$x \cdot A = \delta B + \epsilon C$ & $\delta \neq 0 \text{ or } \epsilon \neq 0$ \\ \hline


4 & $q^{-2}$	& $\begin{pmatrix}q      & q^{-1}  \\ q      & q^{-1}\end{pmatrix}$ &
$x \cdot A = \delta B, \ x \cdot C = \delta D$ & $\delta\neq 0$ \\ \hline

5 & $q^{-2}$	& $\begin{pmatrix}q      & q   \\ q^{-1}      & q^{-1}\end{pmatrix}$ &
$x \cdot A = \delta C, \ x \cdot B = \delta D$ & $\delta\neq 0$ \\ \hline

6 & $q^{-2}$	& $\begin{pmatrix}q^{-1}  & q   \\ q      & q^3\end{pmatrix}$ &
$x \cdot D = \delta B + \epsilon C$ & $\delta \neq 0 \text{ or } \epsilon \neq 0$ \\ \hline


7 & $q^4$	& $\begin{pmatrix}q^{-4}  & q^{-2}  \\ q^{-2}  & 1     \end{pmatrix}$ &
$x \cdot A = \delta D$ & $\delta\neq 0$ \\ \hline

8 & $q^{-4}$	& $\begin{pmatrix}1      & q^2    \\ q^2    & q^4\end{pmatrix}$ &
$x \cdot D = \delta A$ & $\delta\neq 0$ \\ \hline
\end{tabular}
\end{table}
  
  \begin{proof}[Proof of Proposition~\ref{prop:matrixactions}]
    It is straightforward to check that each item in Table \ref{table.mat2} indeed gives an action of $T_n(\lambda, m, 0)$ on $\mc O_q(M_2(\kk))$.
    Now assume we have an action with $g$ and $x$ acting according to the hypotheses.
    We show that this action is one of those listed above.
    Let $(\eta_{ij})$ give the action of $x$ on the generators, i.e. $(\eta_{ij})$ is the matrix representing the action of $x$ on the basis $(A,B,C,D)$ of the $1$-graded piece.
    
    By Lemma~\ref{lem:2matrixdiag}, we have that $g$ must act as an element of $\mc H$, i.e. diagonally.
    Thus, by Lemma~\ref{lem.gxcom}, we have $\eta_{ij} \eta_{ji} = 0$ for all $i,j$, and in particular, each $\eta_{ii} = 0$.
    Moreover, the coefficients of $B^2$ and $C^2$ in $x \cdot (AD - DA - (q - q^{-1}) BC)$ are, respectively, $(q^{-1} - q) \alpha_{12} \eta_{23}$ and $(q^{-1} - q) \eta_{32}$.
    Hence, $\eta_{23} = \eta_{32} = 0$.
    
    By Remark~\ref{rem:iso}, we can assume $x \cdot A \neq 0$ or $x \cdot B \neq 0$:
    If $x \cdot C \neq 0$, we can consider instead the action on $\mc O_q(M_2(\kk))$ from Remark~\ref{rem:iso} with $\phi = \tau$, in which $x \cdot B \neq 0$. If, on the other hand, $x \cdot D \neq 0$, we can consider the action on $\mc O_{q^{-1}}(M_2(\kk))$ with $\phi = A \mapsto D, B \mapsto C, C \mapsto B, D \mapsto A$, in which $x \cdot A \neq 0$.
    First, suppose $x \cdot B \neq 0$, so $\eta_{12} \neq 0$ or $\eta_{42} \neq 0$.
    By Remark~\ref{rem:iso} again, we can assume $\eta_{12} \neq 0$.
    Then by \eqref{eq.coeff}, $\lambda = \frac{\alpha_{11}}{\alpha_{12}}$.
    Also, since the coefficient of $A^2$ in $x \cdot (AB - qBA)$ is $\eta_{12}(\alpha_{11} - q)$, we get that $\alpha_{11} = q$.
    Similarly, the coefficient of $AC$ in $x \cdot (BC - CB)$ is $\eta_{12} (1 - q^{-1} \alpha_{21})$, so $\alpha_{21} = q$.
    Finally, the coefficient of $AD$ in $x \cdot (BD - qDB)$ is $\eta_{12} (1 - q \alpha_{22})$, so $\alpha_{22} = q^{-1}$.
    Using the fact that $g$ has rank one, we have $\alpha_{12} = q^{-1}$ as well and $\lambda = \frac{\alpha_{11}}{\alpha_{12}} = q^2$.
    Since we also assume $\lambda^2 \neq 1$, by \eqref{eq.coeff}, we have $\eta_{13} = \eta_{21} = \eta_{24} = \eta_{31} = \eta_{41} = \eta_{42} = \eta_{43} = 0$.
    The $A^2$ and $AC$ coefficients of $x \cdot (AD - DA - (q-q^{-1})BC)$ are, respectively, $\eta_{14}(\alpha_{11} - 1)$ and $\eta_{34}(\alpha_{11} - q^{-1}) - (q-q^{-1})\eta_{12}$.
    This yields $\eta_{14} = 0$ and $\eta_{34} = \eta_{12}$.
    Thus, $g$ and $x$ act according to the first case of $\lambda = q^2$ in Table~\ref{table.mat2}.
    

    Without loss of generality, we can now assume $x \cdot B = x \cdot C = 0$ and $x \cdot A \neq 0$, so at least one of $\eta_{21}, \ \eta_{31},$ and $\eta_{41}$ must be nonzero.
    Suppose that $\eta_{41} \neq 0$.
    Then by \eqref{eq.coeff}, $\lambda = \frac{\alpha_{22}}{\alpha_{11}}$.
    The coefficients of $D^2$ in $x \cdot (AD - DA - (q - q^{-1})BC)$, $CD$ in $x \cdot (AC - qCA)$, and $BD$ in $x \cdot (AB - qBA)$ are, respectively, $\eta_{41}(1 - \alpha_{22})$, $\eta_{41}(q^{-1} - q \alpha_{21})$ and $\eta_{41} (q^{-1} - q \alpha_{12})$.
    Thus, we have $\alpha_{12} = \alpha_{21} = q^{-2}$, $\alpha_{22} = 1$, and by the fact that $g$ has rank one as a matrix in $\mc H$, we have $\alpha_{11} = q^{-4}$ and $\lambda = q^4$.
    From \eqref{eq.coeff} and our assumptions that $q^2 \neq 1$ and $\lambda^2 \neq 1$, we then know $\eta_{21} = \eta_{31} = \eta_{14} = 0$.
    The $B^2$ coefficient of $x \cdot (BD - qDB)$ and the $C^2$ coefficient of $x \cdot (CD - qDC)$ are respectively $\eta_{24}(q^{-2} - q)$ and $\eta_{34}(q^{-2} - q)$.
    If $q^3 \neq 1$, we have $\eta_{24} = \eta_{34} = 0$, in which case $x \cdot D = 0$ and $x \cdot A = \eta_{41} D$, so $g$ and $x$ act according to the case $\lambda = q^4$ in Table~\ref{table.mat2}.
    If on the other hand, $q^3 = 1$, then $g$ and $x$ act according to Example~\ref{ex:ordqis3}~(2).
     

    
    Now assume $x \cdot B = x \cdot C = 0$, and $x \cdot A = \eta_{21} B + \eta_{31} C \neq 0$.
    Moreover, by the above paragraph and Remark~\ref{rem:iso}, we can assume $\eta_{14} = 0$.
    In the case $\eta_{21} \neq 0$, by \eqref{eq.coeff}, $\lambda = \frac{\alpha_{12}}{\alpha_{11}}$.
    Also, the $B^2$ coefficient of $x\cdot(AB - qBA)$, the $BC$ coefficient of $x\cdot(AC - qCA)$, and the $BD$ coefficient of $x\cdot(AD-DA-(q-q^{-1})BC)$ are, respectively, $\eta_{21}(1 - q \alpha_{12})$, $\eta_{21}(1 - q \alpha_{21})$, and $\eta_{21}(1 - q^{-1}\alpha_{22})$.
    Hence, we have $\alpha_{12} = \alpha_{21} = q^{-1}$ and $\alpha_{22} = q$.
    By the fact that $g$ has rank 1, we also have $\alpha_{11} = q^{-3}$ and hence that $\lambda = q^2$.
    We obtain the same result in the case $\eta_{31} \neq 0$.
    Also, in either case, \eqref{eq.coeff} yields $\eta_{24} = \eta_{34} = 0$, so $g$ and $x$ act according to the final case of $\lambda = q^2$ in Table~\ref{table.mat2}.
  \end{proof}

\begin{corollary}
  Let $q \neq \pm 1$ and assume $m \geq 3$. 
  Then $T_n(\lambda, m, 0)$ acts inner faithfully on $\mc O_q(M_2(\kk))$ with $x$ acting linearly and nonzero, and $g$ acting as an element of $\mc H \rtimes \langle \tau \rangle$, if and only if 
  \begin{itemize}
      \item $\lambda = q^{\pm 2}$ and $\ord(q) = n$, or
      \item $\lambda = q^{\pm 4}$ and $\ord(q^2) = n$.
  \end{itemize}
\end{corollary}

  It is possible to ``patch'' the actions of Proposition~\ref{prop:matrixactions} together to get actions of bosonizations of higher rank quantum linear spaces.
  \begin{example} \label{ex:rank3matrixaction}
    Let $q \in \kk$ be a fifth root of unity and let $G = (\ZZ_5)^3$ with generators $g_1, g_2, g_3$.
    Also, let $\chi_1, \chi_2, \chi_3 \in \wh G$ be defined by
    \begin{equation} \label{eq:b3pair}
      \begin{array}{lll}
         \chi_1(g_1) = q^2, & \chi_1(g_2) = 1, & \chi_1(g_3) = q^{-2}, \\
         \chi_2(g_1) = 1, & \chi_2(g_2) = q^2, & \chi_2(g_3) = q^{-2}, \\
         \chi_3(g_1) = q^2, & \chi_3(g_2) = q^2, & \chi_3(g_3) = q^{-4}.
      \end{array}
    \end{equation}
    Note that $\cu$ and $\gu$ satisfy the necessary conditions to form a quantum linear space $\mc R( \gu, \cu)$.
    The bosonization of this quantum linear space with  the group algebra $\kk G$, namely $B(G, \gu, \cu)$, is generated by grouplike elements $\{g_1, g_2, g_3\}$ and $(g_i, 1)$-skew primitive elements $x_i$ subject to the relations of $G$ and
    \[
      g_i x_j = \chi_j(g_i) x_j g_i, \quad x_i x_j = \chi_j(g_1) x_j x_i.
    \]
    There is an action of $B(G, \gu, \cu)$ on $\mc O_q(M_2(\kk))$ specified by 
    \begin{enumerate}[label=(\arabic*)]
       \item $B_1 \cong T_{5}(q^2, 5, 0)$ acts as in Table \ref{table.mat2} (1) with $\delta_1 \in \kk^\times$ arbitrary,
       \item $B_2 \cong T_{5}(q^2, 5, 0)$ acts as in Table \ref{table.mat2} (2) with $\delta_2 \in \kk^\times$ arbitrary, and
       \item $B_3 \cong T_{5}(q^4, 5, 0)$ acts as in Table \ref{table.mat2} (8) with $\delta_3 \in \kk^\times$ arbitrary.
    \end{enumerate} 
    To see that this indeed defines an action, we need only verify that $g_i x_j - \chi_j(g_i) x_j g_i$ and $x_i x_j - \chi_j(g_1) x_j x_i$ act by zero for $i \neq j$.
    Representing these elements as matrices on the basis $(A, B, C, D)$ of $\mc O_q(M_2(\kk))_1$, one easily verifies that these matrices satisfy the necessary relations.
  \end{example}

  We show in the next theorem that this is the most ``patching'' that can be done for such actions.
  
  \begin{theorem} \label{thm.mat2}
    Let $q \neq \pm 1$.
    Also, let $B(G,\gu, \cu)$ be a bosonization of rank $\theta$ with $m_i \neq 2$ for all $i$.
    Suppose $B(G,\gu, \cu)$  acts on $\mc O_q( M_2(\kk))$ with each $g_i$ acting as an element of $\mc H \rtimes \langle \tau \rangle$ and each $x_i$ acting linearly and nonzero.
    Then, $\theta \leq 3$.
  \end{theorem}

    {
    \setlength{\tabcolsep}{.67em}
    \begin{table}[h]
    \caption{Compatibility of actions of $B_i$ and $B_j$ on $\mc O_q(M_2(\kk))$. If the action of $B_i \cong T_{n_i}(\lambda_i, m_i, 0)$ and $B_j \cong T_{n_j}(\lambda_j, m_j, 0)$ (with generators $g_i, \ x_i$ and $g_j, \ x_j$ respectively) contained in $B(G, \gu, \cu)$ are specified by Proposition~\ref{prop:matrixactions} corresponding to the number in the first row and column respectively, the conditions of the table are necessary and sufficient for the relations $x_i x_j = \zeta_{ji} x_j x_i$, $g_i x_j = \zeta_{ji} x_j g_i$, and $g_j x_i = \zeta_{ji}^{-1} x_i g_j$ to hold. 
    (Here, $\zeta_{ji} = \chi_j(g_i)$.)
    The symbol --- means the actions are always incompatible.} \label{table:m2compat}
    \begin{tabular}{|c|c|c|c|c|c|c|c|c|} 
      \hline
          & 1
          & 2
          & 3
          & 4
          & 5
          & 6
          & 7
          & 8 \\
      \hline
        1
          & ---
          & $\zeta_{ji} = 1$
          & ---
          & ---
          & $\zeta_{ji} = 1$
          & \begin{tabular}{@{}c@{}} 
            $\zeta_{ji} = q^{-2}$, \\ 
            $\delta_i = 0$
            \end{tabular}
          & ---
          & $\zeta_{ji} = q^{-2}$ \\
      \hline
        2
          & $\zeta_{ji} = 1$
          & ---
          & ---
          & $\zeta_{ji} = 1$
          & ---
          & \begin{tabular}{@{}c@{}} 
            $\zeta_{ji} = q^{-2}$, \\ 
            $\epsilon_i = 0$
            \end{tabular}
          & ---
          & $\zeta_{ji} = q^{-2}$ \\
      \hline
        3
          & ---
          & ---
          & ---
          & \begin{tabular}{@{}c@{}} 
            $\zeta_{ji} = q^{-2}$, \\ 
            $\epsilon_j = 0$
            \end{tabular}
          & \begin{tabular}{@{}c@{}} 
            $\zeta_{ji} = q^{-2}$, \\ 
            $\delta_j = 0$
            \end{tabular}
          & $\zeta_{ji} = q^{2}$
          & \begin{tabular}{@{}c@{}}
            $\zeta_{ji} = q^{-4}$, \\
            $q^6 = 1$
            \end{tabular}
          & --- \\
      \hline    
        4
          & ---
          & $\zeta_{ji} = 1$
          & \begin{tabular}{@{}c@{}}
            $\zeta_{ji} = q^2$, \\
            $\epsilon_i = 0$
            \end{tabular}
          & ---
          & $\zeta_{ji} = 1$
          & ---
          & $\zeta_{ji} = q^{2}$
          & --- \\
      \hline   
        5
          & $\zeta_{ji} = 1$
          & ---
          & \begin{tabular}{@{}c@{}}
            $\zeta_{ji} = q^2$, \\
            $\delta_i = 0$
            \end{tabular}
          & $\zeta_{ji} = 1$
          & ---
          & ---
          & $\zeta_{ji} = q^{2}$
          & --- \\
      \hline   
        6
          & \begin{tabular}{@{}c@{}}
            $\zeta_{ji} = q^2$, \\
            $\delta_j = 0$
            \end{tabular}
          & \begin{tabular}{@{}c@{}}
            $\zeta_{ji} = q^2$, \\
            $\epsilon_j = 0$
            \end{tabular}
          & $\zeta_{ji} = q^{-2}$
          & ---
          & ---
          & ---
          & ---
          & \begin{tabular}{@{}c@{}}
            $\zeta_{ji} = q^4$, \\
            $q^6 = 1$
            \end{tabular} \\
      \hline   
        7
          & ---
          & ---
          & \begin{tabular}{@{}c@{}}
            $\zeta_{ji} = q^4$, \\
            $q^6 = 1$
            \end{tabular}
          & $\zeta_{ji} = q^{-2}$
          & $\zeta_{ji} = q^{-2}$
          & ---
          & ---
          & --- \\
      \hline  
        8
          & $\zeta_{ji} = q^{2}$
          & $\zeta_{ji} = q^{2}$
          & ---
          & ---
          & ---
          & \begin{tabular}{@{}c@{}}
            $\zeta_{ji} = q^{-4}$, \\
            $q^6 = 1$
            \end{tabular}
          & ---
          & --- \\
      \hline  
    \end{tabular}
    \end{table}
    }
  
  \begin{proof}
    By Proposition~\ref{prop:matrixactions}, each $B_i$ must act by one of the eight actions of Table~\ref{table.mat2} (or the two actions of Example~\ref{ex:ordqis3} in the case that $\ord(q) = 3$).
    Conditions for compatibility of the actions from Table~\ref{table.mat2} are specified in Table~\ref{table:m2compat}, the contents of which follow from basic computations.
    The table must be symmetric of course (after switching $i \leftrightarrow j$).
    Also, Remark~\ref{rem:iso} minimizes the calculations needed.
    
    For the actions of Example~\ref{ex:ordqis3}, note that if $\gamma = 0$ or $\delta = \epsilon = 0$, the action reduces to one of those in Table~\ref{table.mat2}.
    Thus, for the sake of finding compatibility of actions, we can assume for each of those that $\gamma \neq 0$ and at least one of $\delta, \epsilon$ is nonzero. Simple calculations show that the first is compatible only with action $(6)$ from the table, while the second action is only compatible with $(3)$.
    Thus, since we want to show $\theta \leq 3$, we need not consider these cases any longer and focus solely on those actions in Table~\ref{table.mat2}.

    We now use Table~\ref{table:m2compat} to show that $\theta \leq 3$.
    We construct an undirected graph with eight vertices corresponding to the action ``types'' in Table \ref{table:m2compat} and exactly one edge between vertices if there is a compatible action between those two types. 
    If there is no compatible action, we draw no edge between those two vertices.
    This gives the following.
  
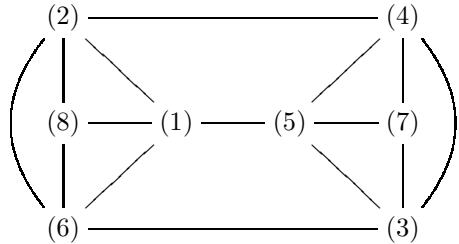
\[
  \xymatrix{
  (2) \ar@{-}[d] \ar@{-}[dr] \ar@{-}@/_{2em}/[dd] \ar@{-}[rrr] &     &       & (4) \ar@{-}[d] \ar@{-}@/^{2em}/[dd] \\
  (8) \ar@{-}[d] \ar@{-}[r] & (1) \ar@{-}[r] & (5) \ar@{-}[r]  \ar@{-}[dr] \ar@{-}[ur] & (7) \ar@{-}[d]\\
  (6) \ar@{-}[ur] \ar@{-}[rrr] &     &       & (3)
  }
  \]
  An action of rank 1 corresponds to a vertex.
  A possible action of rank 2 corresponds to an edge (assuming the compatibility conditions of Table~\ref{table:m2compat} are satisfied).
  A possible action of rank 3 corresponds to a triangle, but not all triangles are valid.
  A possible action of rank 4 corresponds to a $K_4$ subgraph and there are only two of these in the graph.
  One has vertices $(1),(2),(6),(8)$ and the other has vertices $(3),(4),(5),(7)$. 
  We note that ruling out just one of these cases will suffice by Remark~\ref{rem:iso}.
  
  Suppose a rank 4 bosonization $B$ acts on $\mc O_q(M_2(\kk))$ with $B_1$, $B_2$, $B_3$, and $B_4$ acting as (1), (2), (6), and (8) respectively. 
  Then using Table~\ref{table:m2compat}, we must have $\delta_3 = \epsilon_3 = 0$, since the action of $B_3$ must be compatible with both the action of $B_1$ and $B_2$.
  In that case, $x_3$ acts by zero, a contradiction.
  This shows there are no rank 4 actions and hence the highest rank of $B$ is 3.
  \end{proof}

We now turn our attention to the more difficult case of $\mc O_q(M_N(\kk))$ with $N \geq 3$.

\begin{lemma}\label{lem:matrixdiag}
Let $q \neq \pm 1$ and $N,m \geq 3$.
If $T_n(\lambda, m, 0)$ acts on $\mc O_q(M_N(\kk))$ with $x$ acting linearly and nonzero, and $g$ acting as an element of $\mc H \rtimes \langle \tau \rangle$, then $g$ must act as an element of $\mc H$, i.e. diagonally on the basis $(Y_{i,j})$ of $\mc O_q(M_N(\kk))$.
\end{lemma}

\begin{proof}
  Let $\mc N = \{1, 2, \ldots, N\}$.
  First, we write $x \cdot Y_{ij} = \sum_{(a,b) \in \mc N^2} \eta_{ij}^{ab} Y_{ab}$.
  Throughout this proof, we will be using the basis $\{Y_{ij} Y_{k \ell} \mid i < k \text{ or } (i=k \text{ and } j < \ell) \}$ of $\mc O_q(M_N(\kk))_{(2)}$, and will often refer to the $Y_{ij}$ coefficient of a term using this basis.
  Suppose $g$ does not act as an element of $\mc H$, so it must act as $g \cdot Y_{ij} = \alpha_{ij} Y_{ji}$ for $\alpha_{ij} \in \kk^\times$.
  We will show that $x$ must act by zero.
  We have that 
  \begin{equation} \label{eq:matNxcoeffcond}
    0 = (gx - \lambda xg) \cdot Y_{ij} = \sum_{(a,b) \in \mc N^2} \left[ \eta_{ij}^{ba} \; \alpha_{ba} - \lambda \; \alpha_{ij} \; \eta_{ji}^{ab} \right] Y_{ab},
  \end{equation}
  giving that $\eta_{ij}^{ba} = 0$ if and only if $\eta_{ji}^{ab} = 0$ for integers $0 < a,b,i,j \leq N$ satisfying $a \neq b$ or $j \neq i$.
  Therefore, it will suffice to show that $\eta_{ij}^{k\ell} = 0$ whenever $i - j \geq k - \ell$, i.e. when $Y_{k \ell}$ lies on or above the diagonal containing $Y_{ij}$.
  We now show this in steps, using results from earlier steps in later ones without further mention.
  
  \noindent \underline{$k = i, \ \ell = j$:} 
  From \eqref{eq:matNxcoeffcond}, we have $\eta_{ij}^{ij} = \lambda \eta_{ji}^{ji} = \lambda^2 \eta_{ij}^{ij}$.
  Since $\lambda^2 \neq 1$, we must have $\eta_{ij}^{ij} = 0$.

  \noindent \underline{$k < i, \ \ell > j$:} 
  If $k \neq \ell$, then the coefficient of $Y_{k \ell}^2$ in $x \cdot (Y_{ij} Y_{k \ell} - Y_{k \ell} Y_{ij})$ is  $\eta_{ij}^{k \ell}$.
  If $k \neq j$, then the coefficient of $Y_{kj} Y_{k \ell}$ in $x \cdot (Y_{kj} Y_{ij} - q Y_{ij} Y_{kj})$ is $- \eta_{ij}^{k \ell}$.
  
  \noindent \underline{$k < i, \ \ell < j$:} 
  If $k \neq j$, the coefficient of $Y_{k \ell} Y_{k j}$ in $x \cdot (Y_{k j} Y_{i j} - q Y_{i j} Y_{k j})$ is $-q \eta_{ij}^{k \ell}$.
  If $i \neq \ell$, the coefficient of $Y_{k \ell} Y_{i \ell}$ in $x \cdot (Y_{i \ell} Y_{ij} - q Y_{ij} Y_{i \ell})$ is $-q \eta_{ij}^{k \ell}$.
  
  \noindent \underline{$k > i, \ \ell > j$:} 
  Follows similarly to the previous step.
  
  \noindent \underline{$k = i, \ \ell > j$:} 
  If $i \neq \ell$, the coefficient of $Y_{i \ell}^2$ in $x \cdot (Y_{ij} Y_{i \ell} - q Y_{i \ell} Y_{ij})$ is $\eta_{ij}^{i \ell}$.
  If $i = \ell$, the coefficient of $Y_{ji} Y_{ij}$ in $x \cdot (Y_{j j} Y_{i j} - q Y_{i j}Y_{j j})$ is $\eta_{jj}^{ji} - q \alpha_{ij} \eta_{jj}^{ij} + (q^2 - 1) \eta_{ij}^{ii}$. 
  By the case that $i \neq \ell$ and \eqref{eq:matNxcoeffcond}, we have $\eta_{jj}^{ji} = \eta_{jj}^{ij} = 0$.
  Thus, we have $\eta_{ij}^{ii} = 0$ in this case as well.
  
  \noindent \underline{$k < i, \ \ell = j$:} 
  Follows similarly to the previous step.
\end{proof}
  
\begin{table}
\caption[matn]{Actions of $T_n(\lambda, m, 0)$ on $\mc O_q(M_N(\kk))$. 
The second column lists the value of $\lambda$ in terms of $q$ while the third indicates the action of $g$ as an element of $\mc H$, i.e. a matrix of rank one $(\alpha_{ij})_{i,j}$ so that $g \cdot Y_{ij} = \alpha_{ij} Y_{ij}$.
The fourth column indicates how $x$ acts on the generators with $\delta \in \kk^\times$.
We assume the action is trivial if not listed.}\label{table.matn}
\begin{tabular}{|c|c|c|c|} \hline
& $\lambda$ & action of $g$	& action of $x$ \\ \hline	

1 & $q^2$ 	& $\begin{smatrix}
          1      & \cdots & 1      & q      & q^{-1} & 1      & \cdots & 1 \\
          1      & \cdots & 1      & q      & q^{-1} & 1      & \cdots & 1 \\
          \vdots &        & \vdots & \vdots & \vdots & \vdots &        & \vdots \\
          1      & \cdots & 1      & q      & q^{-1} & 1      & \cdots & 1 \\
        \end{smatrix}$ &   
$x \cdot Y_{a,b} = \delta \; Y_{a, b-1} \forall a$	\\
& &  ($q^{-1}$ at column $b$, $b>1$) & \\ \hline

2 & $q^2$	& 
$\begin{smatrix}
          1      & 1      & \cdots & 1 \\
          \vdots & \vdots &        & \vdots \\
          1      & 1      & \cdots & 1 \\
          q      & q      & \cdots & q \\
          q^{-1} & q^{-1} & \cdots & q^{-1} \\
          1      & 1      & \cdots & 1 \\
          \vdots & \vdots &        & \vdots \\
          1      & 1      & \cdots & 1 \\
\end{smatrix}$ & 
$x \cdot Y_{a,b} = \delta \; Y_{a-1, b} \forall b$ \\
& &  ($q^{-1}$ at row $a$, $a>1$) & \\ \hline

3 & $q^2$	& 
$\begin{smatrix}
          q^{-3} & q^{-2} & \cdots & q^{-2} & q^{-1} \\
          q^{-1} & 1      & \cdots & 1      & q      \\
          \vdots & \vdots &        & \vdots & \vdots  \\
          q^{-1} & 1      & \cdots & 1      & q      \\
        \end{smatrix}$ &
$x \cdot Y_{11} = \delta \; Y_{1N}$ \\ \hline
    
4 & $q^2$	& 
$\begin{smatrix}
          q^{-3}  & q^{-1} & \cdots & q^{-1} \\
          q^{-2}  & 1      & \cdots & 1      \\
          \vdots  & \vdots &        & \vdots \\
          q^{-2}  & 1      & \cdots & 1      \\
          q^{-1}  & q      & \cdots & q      \\
        \end{smatrix}$ &
$x \cdot Y_{11} = \delta \; Y_{N1}$ \\ \hline

5 & $q^{-2}$	& 
$\begin{smatrix}
          1      & \cdots & 1      & q      & q^{-1} & 1      & \cdots & 1 \\
          1      & \cdots & 1      & q      & q^{-1} & 1      & \cdots & 1 \\
          \vdots &        & \vdots & \vdots & \vdots & \vdots &        & \vdots \\
          1      & \cdots & 1      & q      & q^{-1} & 1      & \cdots & 1 \\
        \end{smatrix}$ &
$x \cdot Y_{a,b} = \delta \; Y_{a, b+1} \forall a$ \\
& & ($q$ at column $b$, $b<N$) & \\ \hline

6 & $q^{-2}$	& 
$\begin{smatrix}
          1      & 1      & \cdots & 1 \\
          \vdots & \vdots &        & \vdots \\
          1      & 1      & \cdots & 1 \\
          q      & q      & \cdots & q \\
          q^{-1} & q^{-1} & \cdots & q^{-1} \\
          1      & 1      & \cdots & 1 \\
          \vdots & \vdots &        & \vdots \\
          1      & 1      & \cdots & 1 \\
        \end{smatrix}$ &
$x \cdot Y_{a,b} = \delta \; Y_{a+1, b} \forall b$ \\
& & ($q$ at row $a$, $a<N$) & \\ \hline
    
7 & $q^{-2}$	& 
$\begin{smatrix}
          q^{-1} & \cdots & q^{-1} & q      \\
          1      & \cdots & 1      & q^2    \\
          \vdots &        & \vdots & \vdots \\
          1      & \cdots & 1      & q^2    \\
          q      & \cdots & q      & q^3    \\
        \end{smatrix}$ &
$x \cdot Y_{NN} = \delta \; Y_{1N}$ \\ \hline

8 & $q^{-2}$	& 
$\begin{smatrix}
          q^{-1} & 1      & \cdots & 1      & q      \\
          \vdots & \vdots &        & \vdots & \vdots \\
          q^{-1} & 1     & \cdots  & 1      & q \\
          q      & q^2   & \cdots  & q^2    & q^3   \\
\end{smatrix}$ &
$x \cdot Y_{NN} = \delta \; Y_{N1}$ \\ \hline
\end{tabular}
\end{table}
  
  \begin{proposition}\label{prop:matrixactions2}
    Let $q \neq \pm 1$ and $N,m \geq 3$.
    Then $T_n(\lambda, m, 0)$ acts on $\mc O_q(M_N(\kk))$ with $x$ acting linearly and nonzero, and $g$ acting as an element of $\mc H \rtimes \langle \tau \rangle$, if and only if $\lambda = q^{\pm 2}$ and $\ord(q) \mid n$.
    The actions are given by Table \ref{table.matn}.
 \end{proposition}
  
  \begin{proof}
    Again, it is straightforward to check that each of the rows in Table \ref{table.matn} defines an action on $\mc O_q(M_N(\kk))$. 
    
    By Lemma~\ref{lem:matrixdiag}, we have that $g$ must act diagonally.
    For convenience, we rewrite \eqref{eq.coeff} in this case.
    Let $\mc N = \{1, 2, \ldots, N\}$ again.
    As before, we write $g \cdot Y_{ij} = \alpha_{ij} Y_{ij}$ and $x \cdot Y_{ij} = \sum_{(a,b) \in \mc N^2} \eta_{ij}^{ab} Y_{ab}$.
    From 
    \[
      0 = (gx - \lambda xg) \cdot Y_{ij} = \sum_{(a,b) \in \mc N^2} \eta_{ij}^{ab}(\alpha_{ab} - \lambda \alpha_{ij}) Y_{ab},
    \]
    we see that for each $(i,j), (a,b) \in \mc N^2$,
    \begin{equation} \label{eq:skewcommute1}
      \eta_{ij}^{ab} = 0 \quad \text{or} \quad \alpha_{ab} = \lambda \alpha_{ij}.
    \end{equation}
    In particular, we see that $\eta_{ij}^{ij} = 0$ for any $(i,j) \in \mc N^2$. 
    
    As in Lemma~\ref{lem:matrixdiag}, we will be using the basis $\{Y_{ij} Y_{k \ell} \mid i < k \text{ or } (i=k \text{ and } j < \ell) \}$ of $\mc O_q(M_N(\kk))_{(2)}$, and will often refer to the $Y_{ij}$ coefficient of a term using this basis.
    
    We will say $Y_{ab}$ and $Y_{ij}$ form an $AD$ pair if $a<i$ and $b<j$.
    We say that $Y_{ab}$ and $Y_{ij}$ form a $BC$ pair if $a<i$ and $b>j$, and similarly for other pairs of generators from $\mc O_q(M_2(\kk))$.
    
    Fix a $BC$ pair, $(i,j)$ and $(k,\ell)$. 
    Then the coefficient of $Y_{ij}^2$ in $x \cdot (Y_{i \ell} Y_{kj} - Y_{kj} Y_{i \ell} - (q - q^{-1}) Y_{ij} Y_{k \ell})$ is $(q^{-1} - q) \alpha_{ij} \eta_{k \ell}^{ij}$.
    Therefore, $\eta_{k \ell}^{ij} = 0$.
    Similarly, since the coefficient of $Y_{k \ell}^2$ is $(q^{-1} - q) \eta_{ij}^{k \ell}$, we have $\eta_{ij}^{k \ell} = 0$.
    
    Fix $(a,b)$ and choose $(i,j)$ and $(k, \ell)$ so that
    \begin{itemize}
    \item 
      $(i,j)$ and $(k,\ell)$ form an $AD$ or $DA$ pair,
    \item
      $(a,b)$ and $(k,\ell)$ form any pair besides $AD$ or $DA$, so $Y_{ab} Y_{k \ell} = \gamma Y_{k \ell} Y_{ab}$ for $\gamma \in \kk^\times$, and
    \item
      $(a,b)$ does not form the corresponding $B$ or $C$ for the $AD$ pair above, i.e. $(a,b) \notin \{(i,\ell), (k,j)\}$.
    \end{itemize}
    The coefficient of $Y_{i \ell} Y_{k j}$ in $x \cdot (Y_{ab}Y_{k \ell} - \gamma Y_{k \ell} Y_{ab})$ is $\wh \gamma \eta_{ab}^{ij}$ for some $\wh \gamma \in \kk^\times$, forcing $\eta_{ab}^{ij} = 0$.
    
    Thus, taking the calculations above, we see that for a fixed $(a,b)$, many $\eta_{ab}^{ij}$ must be zero, as shown in Figure~\ref{fig:xcoeff}.
    
    \begin{figure}[ht!]
       \begin{tikzpicture}[x = 3mm, y=3mm]
         \fill (-3,-3) rectangle (4,4);
         \fill[white] (-1,-1) rectangle (2,2);
         \fill[red] (0,0) rectangle (1,1);
         \fill (1,1) rectangle (4,4);
         \fill (1,0) rectangle (4,-3);
         \fill (0,1) rectangle (-3,4);
         \fill (0,0) rectangle (-3,-3);
         \draw (4,4) -- (4,-3) -- (-3,-3) -- (-3,4) -- (4,4);
       \end{tikzpicture}
       \quad
       \begin{tikzpicture}[x = 3mm, y=3mm]
             \fill (-3,-6) rectangle (4,1);
             \fill[white] (-1,-1) rectangle (4,1);
             \fill[red] (0,0) rectangle (1,1);
             \fill (1,0) rectangle (4,-6);
             \fill (0,0) rectangle (-3,-6);
             \draw (4,1) -- (4,-6) -- (-3,-6) -- (-3,1) -- (4,1);
       \end{tikzpicture}
       \quad
       \begin{tikzpicture}[x = 3mm, y=3mm]
         \fill (-6,-6) rectangle (1,1);
         \fill[white] (-1,-1) rectangle (1,1);
         \fill[red] (0,0) rectangle (1,1);
         \fill (0,0) rectangle (-6,-6);
         \draw (1,1) -- (1,-6) -- (-6,-6) -- (-6,1) -- (1,1);
       \end{tikzpicture}
       \quad
       \begin{tikzpicture}[x = 3mm, y=3mm]
         \fill[red] (0,0) rectangle (1,1);
         \fill (2,0) rectangle (7,-6);
         \fill (1,-1) rectangle (7,-6);
         \draw (7,1) -- (7,-6) -- (0,-6) -- (0,1) -- (7,1);
       \end{tikzpicture}
       
       \caption{
       In each case of a location of $(a,b)$, given by the red square, the black squares represent $(i,j)$ such that $\eta_{ab}^{ij}$ must be $0$ from our calculations.
       For example, if $1 < a,b < N$, then $\eta_{ab}^{ij} = 0$ if $(i,j)$ is not horizontally or vertically adjacent to $(a,b)$.
       The cases for the remaining locations of $(a,b)$ are covered by Remark~\ref{rem:iso}.
       } \label{fig:xcoeff}
    \end{figure}
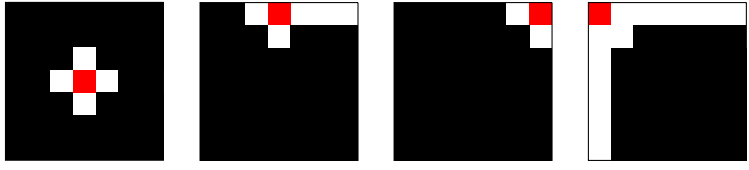
    
    Assume $\eta_{ab}^{cd} \neq 0$ for some $(a,b) \neq (c,d) \in \mc N^2$.
    Choose $(i,j)$ such that 
    \begin{itemize}
    \item 
      $(a,b)$ and $(i,j)$ form any pair besides $AD$ or $DA$, so $Y_{ab}Y_{ij} = \gamma Y_{ij}Y_{ab}$ for $\gamma \in \kk^\times$, and
    \item
      $(c,d)$ and $(i,j)$ form any pair besides $BC$ or $CB$, or altenatively, $(c,d)=(i,j)$.
    \end{itemize}
    Then the coefficient of $Y_{cd} Y_{ij}$ in $x \cdot (Y_{ab} Y_{ij} - \gamma Y_{ij} Y_{ab})$ is $\eta_{ab}^{cd} (1 - \gamma \alpha_{ij} \wh q)$, where
    \begin{align} 
      \wh q &= \begin{cases}\label{eq:qhatgamma}
        1, & \text{$(c,d)$ and $(i,j)$ form an $AD$ or $DA$ pair, or $(c,d) = (i,j)$} \\
        q, & \text{$(c,d)$ and $(i,j)$ form a $BA$ or $CA$ pair} \\
        q^{-1}, & \text{$(c,d)$ and $(i,j)$ form an $AB$ or $AC$ pair,}
      \end{cases} \\
      \gamma &= \begin{cases}
        1, & \text{$(a,b)$ and $(i,j)$ form a $BC$ or $CB$ pair} \\
        q, & \text{$(a,b)$ and $(i,j)$ form an $AB$ or $AC$ pair} \\
        q^{-1}, & \text{$(a,b)$ and $(i,j)$ form a $BA$ or $CA$ pair.}
      \end{cases}
    \end{align}
    Thus, since $\eta_{ab}^{cd} \neq 0$, we have
    \begin{equation} \label{eq:alphavals}
      \alpha_{ij} = (\gamma \wh{q})^{-1}.
    \end{equation}
    On the other hand, if $(i,j)$ is selected so that it forms the same pair with both $(a,b)$ and $(c,d)$, an $AD$, $DA$, $BC$, or $CB$ pair, then the coefficient of $Y_{ij} Y_{cd}$ in $x \cdot (Y_{ij} Y_{ab} - Y_{ab}Y_{ij} - \gamma Y_{ib} Y_{aj})$ is $\eta_{ab}^{ij} ( \alpha_{ij} - 1)$.
    Thus, in this case,
    \begin{equation} \label{eq:alphavalssamepair}
      \alpha_{ij} = 1.
    \end{equation}
    
    Now suppose $x \cdot Y_{ab} \neq 0$ for some $1 < a,b < N$.
    By Remark~\ref{rem:iso}, we can assume without loss of generality that $\eta_{ab}^{a,b-1} \neq 0$.
    By \eqref{eq:skewcommute1}, \eqref{eq:alphavals}, \eqref{eq:alphavalssamepair}, and the fact that $g$ has rank one, we get that $g$ is as in the first case of $\lambda = q^2$ with $b < N$.
    To see that $x$ must act as specified, first consider $c<a$.
    The coefficient of $Y_{c(b-1)} Y_{a(b-1)}$ in $x \cdot (Y_{c(b-1)} Y_{ab} - Y_{ab} Y_{c (b-1)} - (q-q^{-1})Y_{cb}Y_{a(b-1)})$ is $(\alpha_{c(b-1)} - q^{-1}) \eta_{ab}^{a(b-1)} - (q - q^{-1}) \eta_{cb}^{c(b-1)}$.
    Thus, since $\alpha_{c(b-1)} = q$, we have $\eta_{cb}^{c(b-1)} = \eta_{ab}^{a(b-1)}$.
    Similarly, for $c > a$, the coefficient of $Y_{ab} Y_{c(b-1)}$ in $x \cdot (Y_{ab} Y_{cb} - q Y_{cb}Y_{ab})$ is $(q^{-1} - q) (\eta_{cb}^{c(b-1)} - \eta_{ab}^{a(b-1)})$.
    Thus, in this case also, $\eta_{cb}^{c(b-1)} = \eta_{ab}^{a(b-1)}$.
    That all other $\eta$ coefficients must be zero follows from Figure~\ref{fig:xcoeff}, \eqref{eq:skewcommute1}, and our knowledge of $g$.
    
    From now on, we assume $x \cdot Y_{ab} = 0$ for all $1 < a,b < n$.
    We proceed by assuming that $x \cdot Y_{ab} \neq 0$ for some $(a,b)$ matching the red square in the remaining three cases of Figure~\ref{fig:xcoeff}
    By Remark~\ref{rem:iso}, this is sufficient.
    
    Fix $1 < b < N$ and suppose $x \cdot Y_{1b} \neq 0$.
    Since $\eta_{2b}^{2(b-1)} = 0$, the coefficient of $Y_{1b} Y_{2(b-1)}$ in $x \cdot (Y_{1b} Y_{2b} - q Y_{2b}Y_{1b})$ is $(q^2 - 1) \alpha_{2b} \eta_{1b}^{1(b-1)}$, giving that $\eta_{1b}^{1(b-1)} = 0$.
    Suppose $\eta_{1b}^{1d} \neq 0$ for $d > b$.
    Then the coefficient of $Y_{1d}^2$ in $x \cdot (Y_{1b} Y_{1d} - q Y_{1d} Y_{1b})$ is $\eta_{1b}^{1d} (1 - q \alpha_{1d})$, giving that $\alpha_{1d} = q^{-1}$.
    Also, since $\eta_{2b}^{2d} =0$, the coefficient of $Y_{1d} Y_{2b}$ in $x \cdot (Y_{1b} Y_{2b} - q Y_{2b} Y_{1b})$ is $\eta_{1b}^{1d} (1 - q \alpha_{2b})$, so $\alpha_{2b} = q^{-1}$.
    By \eqref{eq:alphavals} and \eqref{eq:alphavalssamepair} respectively, we know $\alpha_{11} = \alpha_{21} = 1$.
    Thus, since $g$ has rank one, we have $\alpha_{1b} = q^{-1}$.
    But then, by \eqref{eq:skewcommute1}, since $\eta_{1b}^{1d} \neq 0$, we have $\lambda = 1$, a contradiction.
    Thus, $\eta_{1b}^{1d} = 0$ for all $d > b$.
    By the second case of Figure~\ref{fig:xcoeff}, we must have $x \cdot Y_{1b} = \delta Y_{2b}$ for some nonzero $\delta \in \kk$.
    By \eqref{eq:alphavals}, \eqref{eq:alphavalssamepair}, and the fact that $g$ has rank one, we get that $g$ is as in the second case of $\lambda = q^{-2}$ with $a = 1$.
    To see that $x$ must act as specified, first consider $d < b$. 
    The coefficient of $Y_{1b} Y_{2d}$ in $x \cdot (Y_{1d} Y_{1b} - qY_{1b}Y_{1d})$ is $(1 - q^2)(\eta_{1d}^{2d} - \delta)$, so $\eta_{1d}^{2d} = \delta$. This is similar for $d>b$.
    All other $\eta$ coefficients must be zero by Figure~\ref{fig:xcoeff}, \eqref{eq:skewcommute1}, and the action of $g$.

    Now, we can assume $x \cdot Y_{ab} = 0$ for all $(a,b) \notin \{(1,1),\ (1,N),\ (N,1),\ (N,N)\}$.
    Since $\eta_{2N}^{2(N-1)} = 0$, the coefficient of $Y_{1N} Y_{2(N-1)}$ in $x \cdot (Y_{1N} Y_{2N} - q Y_{2N} Y_{1N})$ is $(q^2 - 1) \alpha_{2N} \eta_{1N}^{1(N-1)}$, giving $\eta_{1N}^{1(N-1)} = 0$.
    Similarly, 
    $\eta_{1N}^{2N} = 0$.
    Thus, by Figure~\ref{fig:xcoeff}, $x \cdot Y_{1N} = 0$.
    By Remark~\ref{rem:iso}, $x \cdot Y_{N1} = 0$ as well.
    Thus, we can assume $x \cdot Y_{ab} = 0$ for all $(a,b) \notin \{(1,1),\ (N,N)\}$.
    
    Assume $x \cdot Y_{11} \neq 0$.
    Since $N>2$, the coefficient of $Y_{2N} Y_{N2}$ in $x \cdot (Y_{11} Y_{NN} - Y_{NN} Y_{11} - (q - q^{-1}) Y_{1N} Y_{N1})$ is $(q - q^{-1}) \alpha_{NN} \eta_{11}^{22}$, so $\eta_{11}^{22} = 0$.
    Fix $1 < b < N$.
    The coefficient of $Y_{1N} Y_{2b}$ in $x \cdot (Y_{11} Y_{2N} - Y_{2N} Y_{11} - (q - q^{-1}) Y_{1N} Y_{21})$ is $(q - q^{-1}) \alpha_{2N} \eta_{11}^{1b}$, so $\eta_{11}^{1b} = 0$ for $1 < b < N$.
    By Remark~\ref{rem:iso}, $\eta_{11}^{a1} = 0$ for $1 < a < N$ as well.
    Suppose $\eta_{11}^{1N} \neq 0$.
    Then for $1 < a$, the coefficient of $Y_{1N} Y_{aN}$ in $x \cdot (Y_{11} Y_{aN} - Y_{aN} Y_{11} - (q - q^{-1}) Y_{1N} Y_{a1})$ is $\eta_{11}^{1N} (1 - \alpha_{aN} q^{-1})$.
    Thus, $\alpha_{aN} = q$ for $1 < a$.
    Similarly, we see that $\alpha_{a1} = q^{-1}$ for $1 < a$.
    From the above calculations, as well as \eqref{eq:alphavals}, we see that $g$ is as in the third case of $\lambda = q^2$.
    The coefficient of $Y_{21} Y_{N1}$ in $x \cdot (Y_{11} Y_{21} - q Y_{21} Y_{11})$ is $\eta_{11}^{N1} (q^{-1} - 1)$, so $\eta_{11}^{N1} = 0$.
    Thus, $x$ acts as specified.
    
    If instead of $\eta_{11}^{1N} \neq 0$, we have $\eta_{11}^{N1} \neq 0$, this case reduces to the above by Remark~\ref{rem:iso}.
    This exhausts all possibilities of nonzero actions of $x$.
  \end{proof}

As in the case $N=2$, it is possible to ``patch'' the actions of Proposition~\ref{prop:matrixactions2} together to get actions of bosonizations of higher rank quantum linear spaces.

  \begin{example} \label{ex:matrixaction}
    Fix $N \geq 3$.
    Let $q \in \kk$ be a fifth root of unity and let $G = (\ZZ_5)^{2N-2}$ with generators $g_1, g_2, \ldots, g_{2N-2}$.
    Toward defining $\chi_1, \chi_2, \ldots, \chi_{2N-2} \in \wh G$, we first define the set $\mc S \subset \NN^2$ by
    \begin{gather*}
      \mc S = \{(k,\ell) \mid 1 \leq \ell = k-N+1 \leq N-1\} \cup \left\{(k,\ell) \; \middle | \; 2 \leq \ell = k-N+2 \leq N-1, ~ \ell \neq \frac{N+1}2 \right \}
    \end{gather*}
    if $N$ is odd and
     \begin{align*}
      \mc S = & \left\{(k,\ell) \; \middle | \; N+ 1 \leq k \leq N+ \frac N 2 -1, ~ \ell \in \{k-N, k-N+1 \} \right\} \\
      & \bigcup \left\{(k,\ell) \; \middle | \; N + \frac N 2 \leq k \leq 2N-2, ~ \ell \in \{ k - N + 1, k - N + 2 \} \right \}
    \end{align*}
    if $N$ is even.
    
    Now, let $\chi_j \in \wh G$ for $1 \leq j \leq 2N-2$ be defined by
    \[
      \chi_j(g_i) = \begin{cases}
        q^2, & \text{ if } i = j \leq N-1 \\
        q^{-2}, & \text{ if } i = j > N-1 \\
        q, & \text{ if } (j,i) \in \mc S \\
        q^{-1}, & \text{ if } (i,j) \in \mc S \\
        1, & \text{ otherwise.}
      \end{cases}
    \]
    Note that $\cu$ and $\gu$ satisfy the necessary conditions to form a quantum linear space $\mc R( \gu, \cu)$. Let $B(G, \gu, \cu)$ be the bosonization $\mc R\# \kk G$.
    If $N$ is odd, there is an action of $B(G, \gu, \cu)$ on $\mc O_q(M_N(\kk))$ specified by 
    \begin{enumerate}[label=(\alph*)]
       \item for $1 \leq i \leq \frac{N-1}2$, we have $B_i \cong T_{5}(q^2, 5, 0)$ acts as in (1) of Proposition~\ref{prop:matrixactions2} with $b_i = 2i$ and $\delta_i \in \kk^\times$ arbitrary,
       \item for $\frac{N-1}2 < i \leq N-1$, we have $B_i \cong T_{5}(q^2, 5, 0)$ acts as in (2) of Proposition~\ref{prop:matrixactions2} with $a_i = 2 \left(i - \frac{N-1} 2 \right)$ and $\delta_i \in \kk^\times$ arbitrary, 
       \item for $N-1 < i \leq \frac{3(N-1)}2$, we have $B_i \cong T_{5}(q^{-2}, 5, 0)$ acts as in (5) of Proposition~\ref{prop:matrixactions2} with $b_i = 2(i - N+1)$ and $\delta_i \in \kk^\times$ arbitrary, and
       \item for $\frac{3(N-1)}2 < i \leq 2(N - 1)$, we have $B_i \cong T_{5}(q^{-2}, 5, 0)$ acts as in (6) of Proposition~\ref{prop:matrixactions2} with $a_i = 2(i - \frac{3(N-1)}2)$ and $\delta_i \in \kk^\times$ arbitrary.
    \end{enumerate} 
    On the other hand, if $N$ is even, the action is specified by
    \begin{enumerate}[label=(\alph*)]
       \item for $1 \leq i \leq \frac{N}2$, we have $B_i \cong T_{5}(q^2, 5, 0)$ acts as in (1) of Proposition~\ref{prop:matrixactions2} with $b_i = 2i$ and $\delta_i \in \kk^\times$ arbitrary,
       \item for $\frac{N}2 < i \leq N$, we have $B_i \cong T_{5}(q^2, 5, 0)$ acts as in (2) of Proposition~\ref{prop:matrixactions2} with $a_i = 2 \left(i - \frac{N} 2 \right)$ and $\delta_i \in \kk^\times$ arbitrary, 
       \item for $N < i \leq N + \frac{N-2}2$, we have $B_i \cong T_{5}(q^{-2}, 5, 0)$ acts as in (5) of Proposition~\ref{prop:matrixactions2} with $b_i = 2(i - N)$ and $\delta_i \in \kk^\times$ arbitrary, and
       \item for $N + \frac{N-2}2 < i \leq 2N - 2$, we have $B_i \cong T_{5}(q^{-2}, 5, 0)$ acts as in (6) of Proposition~\ref{prop:matrixactions2} with $a_i = 2(i - N - \frac{N-2}2)$ and $\delta_i \in \kk^\times$ arbitrary.
    \end{enumerate} 
    In either case, to see that this indeed defines an action, we need only verify that $g_i x_j - \chi_j(g_i) x_j g_i$ and $x_i x_j - \chi_j(g_i) x_j x_i$ act by zero for $i \neq j$.
  \end{example}

  As for the case $N=2$ (Theorem~\ref{thm.mat2}), the above examples give the most ``patching'' that can be done for such actions.
  
  \begin{theorem} \label{thm.matn}
    Let $q \neq \pm 1$.
    Also, let $B(G,\ul g, \ul \chi)$ be a bosonization of rank $\theta$ with $m_i \neq 2$ for all $i$.
    Suppose $B(G,\ul g, \ul \chi)$  acts on $\mc O_q( M_N(\kk))$ for $N \geq 3$ with each $g_i$ acting as an element of $\mc H \rtimes \langle \tau \rangle$ and each $x_i$ acting linearly and nonzero.
    Then, $\theta \leq 2N - 2$.
  \end{theorem}
  
  \begin{proof}
    By Proposition~\ref{prop:matrixactions2}, each $B_i$ must act by one of the eight actions specified therein.
    Conditions for compatibility of actions of $B_i$ and $B_j$ ($i \neq j$) are specified in Table~\ref{table:mNcompat}, the contents of which follow from basic computations.
    The table must be symmetric of course (after switching $i \leftrightarrow j$ and $\zeta \leftrightarrow \zeta^{-1}$).
    Also, Remark~\ref{rem:iso} minimizes the calculations needed.

    \begin{table}
    \caption{Compatibility of actions of $B_i$ and $B_j$ on $\mc O_q(M_N(\kk))$. 
    If the action of $B_i \cong T_{n_i}(\lambda_i, m_i, 0)$ and $B_j \cong T_{n_j}(\lambda_j, m_j, 0)$ (with generators $g_i, \ x_i$ and $g_j, \ x_j$ respectively) contained in $B(G, \ul{g}, \ul{\chi})$ are specified by Proposition~\ref{prop:matrixactions2} corresponding to the number in the first row and column respectively, the conditions of the table are necessary and sufficient for the relations $x_i x_j = \zeta_{ji} x_j x_i$, $g_i x_j = \zeta_{ji} x_j g_i$, and $g_j x_i = \zeta_{ji}^{-1} x_i g_j$ to hold. 
    Here, $\zeta_{ji} = \chi_j(g_i)$, and for $\alpha \in \kk$ and a set $\mc S$, we let $\alpha^{\mc S}_x$ denote $\alpha$ if $x \in \mc S$ and $1$ otherwise.}  \label{table:mNcompat}
    { \setlength{\tabcolsep}{.67em}
    \begin{tabular}{|c|c|c|c|c|c|} 
      \hline
          & 1
          & 2
          & 3
          & 4
          & $\cdots$ \\
      \hline
        1
          & \begin{tabular}{@{}c@{}} 
            $|b_i - b_j| > 1$, \\ 
            $\zeta_{ji} = 1$
            \end{tabular}
          & $\zeta_{ji} = 1$
          & \begin{tabular}{@{}c@{}} 
            $2 < b_j < N$, \\ 
            $\zeta_{ji} = 1$
            \end{tabular}
          & \begin{tabular}{@{}c@{}} 
            $2 < b_j \leq N$, \\ 
            $\zeta_{ji} = 1$
            \end{tabular}
          & $\cdots$ \\
      \hline
        2
          & $\zeta_{ji} = 1$
          & \begin{tabular}{@{}c@{}} 
            $|a_i - a_j| > 1$, \\ 
            $\zeta_{ji} = 1$
            \end{tabular}
          & \begin{tabular}{@{}c@{}} 
            $2 < a_j \leq N$, \\ 
            $\zeta_{ji} = 1$
            \end{tabular}
          & \begin{tabular}{@{}c@{}} 
            $2 < b_j < N$, \\ 
            $\zeta_{ji} = 1$
            \end{tabular}
          & $\cdots$ \\
      \hline
        3
          & \begin{tabular}{@{}c@{}} 
            $2 < b_i < N$, \\ 
            $\zeta_{ji} = 1$
            \end{tabular}
          & \begin{tabular}{@{}c@{}} 
            $2 < a_i \leq N$, \\ 
            $\zeta_{ji} = 1$
            \end{tabular}
          & ---
          & ---
          & $\cdots$ \\
      \hline    
        4
          & \begin{tabular}{@{}c@{}} 
            $2 < b_i \leq N $, \\ 
            $\zeta_{ji} = 1$
            \end{tabular}
          & \begin{tabular}{@{}c@{}} 
            $2 < a_i < N$, \\ 
            $\zeta_{ji} = 1$
            \end{tabular}
          & --- 
          & --- 
          & $\cdots$ \\
      \hline   
        5
          & \begin{tabular}{@{}c@{}} 
            $b_i \neq b_j + 1$, \\ 
            $\zeta_{ji} = q_{b_i}^{\{b_j, b_j + 2\}}$
            \end{tabular}
          & $\zeta_{ji} = 1$
          & $\zeta_{ji} = q_{b_j}^{\{1,N-1\}}$
          & \begin{tabular}{@{}c@{}} 
            $1 < b_j \leq N-1$, \\ 
            $\zeta_{ji} = 1$
          \end{tabular}
          & $\cdots$ \\
      \hline   
        6
          & $\zeta_{ji} = 1$
          & \begin{tabular}{@{}c@{}} 
            $a_i \neq a_j + 1$, \\ 
            $\zeta_{ji} = q_{a_i}^{\{a_j, a_j + 2\}}$
            \end{tabular}
          & \begin{tabular}{@{}c@{}} 
            $1 < a_j \leq N-1$, \\ 
            $\zeta_{ji} = 1$
            \end{tabular}
          & $\zeta_{ji} = q_{a_j}^{\{1,N-1\}}$
          & $\cdots$ \\
      \hline   
        7
          & \begin{tabular}{@{}c@{}} 
            $2 \leq b_i < N$, \\ 
            $\zeta_{ji} = 1$
            \end{tabular}
          & $\zeta_{ji} = q_{a_i}^{\{2,N\}}$
          & $\zeta_{ji} = q^{-2}$
          & $\zeta_{ji} = q^{-2}$
          & $\cdots$ \\
      \hline  
        8
          & $\zeta_{ji} = q_{b_i}^{\{2,N\}}$ 
          & \begin{tabular}{@{}c@{}} 
            $2 \leq a_i < N$, \\ 
            $\zeta_{ji} = 1$
            \end{tabular}
          & $\zeta_{ji} = q^{-2}$
          & $\zeta_{ji} = q^{-2}$
          & $\cdots$ \\
      \hline  
    \end{tabular}
    
    \bigskip
    
    \begin{tabular}{|c|c|c|c|c|c|} 
      \hline
          & $\cdots$
          & 5
          & 6
          & 7
          & 8 \\
      \hline
        1
          & $\cdots$
          & \begin{tabular}{@{}c@{}} 
            $b_j \neq b_i + 1$, \\ 
            $\zeta_{ji} = (q^{-1})^{\{b_i, b_i + 2\}}_{b_j}$
            \end{tabular}
          & $\zeta_{ji} = 1$
          & \begin{tabular}{@{}c@{}}
            $2 \leq b_j < N$,  \\
            $\zeta_{ji} = 1$ 
            \end{tabular}
          & $\zeta_{ji} = (q^{-1})_{b_j}^{\{2, N\}}$ \\
      \hline
        2
          & $\cdots$
          & $\zeta_{ji} = 1$
          & \begin{tabular}{@{}c@{}} 
            $a_j \neq a_i + 1$, \\ 
            $\zeta_{ji} = (q^{-1})_{a_j}^{\{a_i, a_i + 2\}}$
            \end{tabular}
          & $\zeta_{ji} = (q^{-1})_{a_j}^{\{2, N\}}$
          & \begin{tabular}{@{}c@{}}
            $2 \leq a_j < N$,  \\
            $\zeta_{ji} = 1$ 
            \end{tabular} \\
      \hline
        3
          & $\cdots$
          & $\zeta_{ji} = (q^{-1})_{b_i}^{\{1, N-1\}}$
          & \begin{tabular}{@{}c@{}} 
            $1 < a_i \leq N-1 $, \\ 
            $\zeta_{ji} = 1$
            \end{tabular}
          & $\zeta_{ji} = q^2$
          & $\zeta_{ji} = q^2$ \\
      \hline    
        4
          & $\cdots$
          & \begin{tabular}{@{}c@{}} 
            $1 < b_i \leq N-1$, \\ 
            $\zeta_{ji} = 1$
            \end{tabular}
          & \begin{tabular}{@{}c@{}} 
            $\zeta_{ji} = $ \\
            $(q^{-1})_{a_i}^{\{1, N-1\}}$
            \end{tabular}
          & $\zeta_{ji} = q^{2}$
          & $\zeta_{ji} = q^{2}$ \\
      \hline   
        5
          & $\cdots$
          & \begin{tabular}{@{}c@{}} 
            $|b_i - b_j| > 1$, \\ 
            $\zeta_{ji} = 1$
            \end{tabular}
          & $\zeta_{ji} = 1$
          & \begin{tabular}{@{}c@{}} 
            $1 \leq b_j < N - 1$, \\ 
            $\zeta_{ji} = 1$
          \end{tabular}
          & \begin{tabular}{@{}c@{}} 
            $1 < b_j < N - 1$, \\ 
            $\zeta_{ji} = 1$
          \end{tabular} \\
      \hline   
        6
          & $\cdots$
          & $\zeta_{ji} = 1$
          & \begin{tabular}{@{}c@{}} 
            $|a_i - a_j| > 1$, \\ 
            $\zeta_{ji} = 1$
            \end{tabular}
          & \begin{tabular}{@{}c@{}} 
            $1 < a_j < N - 1$, \\ 
            $\zeta_{ji} = 1$
            \end{tabular}
          & \begin{tabular}{@{}c@{}} 
            $1 \leq a_j < N - 1$, \\ 
            $\zeta_{ji} = 1$
            \end{tabular} \\
      \hline   
        7
          & $\cdots$
          & \begin{tabular}{@{}c@{}} 
            $1 \leq b_i < N -1$, \\ 
            $\zeta_{ji} = 1$
            \end{tabular}
          & \begin{tabular}{@{}c@{}} 
            $1 < a_i < N-1$, \\ 
            $\zeta_{ji} = 1$
            \end{tabular}
          & ---
          & --- \\
      \hline  
        8
          & $\cdots$
          & \begin{tabular}{@{}c@{}} 
            $1 < b_i < N-1$, \\ 
            $\zeta_{ji} = 1$
            \end{tabular}
          & \begin{tabular}{@{}c@{}} 
            $1 \leq a_i < N -1$, \\ 
            $\zeta_{ji} = 1$
            \end{tabular}
          & ---
          & --- \\
      \hline  
    \end{tabular}
    }
    \end{table}

    We now use Table~\ref{table:mNcompat} to show that $\theta \leq 2N - 2$.
    First, note that if more than one $B_i$ act as (1) (or (2), (5), or (6)), then the $b$-values (or $a$-values for (2) and (6)) of the corresponding actions must be at least 2 apart.
    Second, note that if a $B_i$ acts as (1) and another acts as (5), then the $b$-value for (1) cannot be exactly one more than the $b$-value for (5).
    Similarly, the $a$-value for any (2) cannot be one more than the $a$-value for any (6).
    Finally, note that if one of the $B_i$ acts as (3), (4), (7), or (8), then without loss of generality, using Remark~\ref{rem:iso}, we can assume $i=1$ and it acts as (3).
    By Table~\ref{table:mNcompat}, no other $B_i$ can act as (3) or (4).
    Also, one of the $B_i$ could act as (7) or (8), but not both.
    Thus, we consider four cases: $B_1$ acts as (3) and $B_2$ acts as (7), $B_1$ acts as (3) and $B_2$ acts as (8), $B_1$ acts as (3) with none of the $B_i$ acting as (7) or (8), and none of the $B_i$ act as (3), (4), (7), or (8).

    \noindent \textbf{Case 1: $B_1$ acts as (3) and $B_2$ acts as (7):}
    In this case, any $B_i$ acting as (1), (2), (5), or (6) must satisfy the following, respectively:
    \[
      3 \leq b \leq N-1, \qquad
      3 \leq a \leq N, \qquad
      1 \leq b \leq N-2, \qquad
      2 \leq a \leq N-2.
    \]
    If $N$ is even, we can have $B_i$ acting as 
    \[ \arraycolsep=.5cm \def\arraystretch{1.2}
       \begin{array}{ll}
          \text{(1) with } b = 3, 5, \ldots, N-1, 
            &\text{(2) with } a = 4, 6, \ldots, N, \\
          \text{(5) with } b = 1, 3, \ldots, N-3, 
            &\text{(6) with } a = 2, 4, \ldots, N-2.
       \end{array} 
    \]
    Thus, including the actions of $B_1$ as (3) and $B_2$ as (7), the largest $\theta$ could be is $2 + 4 \left( \frac {N-2} 2 \right) = 2N - 2$.
    On the other hand, if $N$ is odd, we can have $B_i$ acting as
    \[ \arraycolsep=.5cm \def\arraystretch{1.2}
       \begin{array}{ll}
          \text{(1) with } b = 3, 5, \ldots, N-2, 
            &\text{(2) with } a = 3, 5, \ldots, N, \\
          \text{(5) with } b = 1, 3, \ldots, N-2, 
            &\text{(6) with } a = 3, 5, \ldots, N-2.
       \end{array} 
    \]
    Hence, the largest $\theta$ could be is $2 + 2 \left(\frac{N-3} 2 \right) + 2 \left( \frac {N-1} 2 \right) = 2N - 2$.
    
    \noindent \textbf{Case 2: $B_1$ acts as (3) and $B_2$ acts as (8):}
    In this case, any $B_i$ acting as (1), (2), (5), or (6) must satisfy the following, respectively:
    \[
      3 \leq b \leq N-1, \qquad
      3 \leq a \leq N-1, \qquad
      2 \leq b \leq N-2, \qquad
      2 \leq a \leq N-2.
    \]
    If $N$ is even, we can have $B_i$ acting as 
    \[ \arraycolsep=.5cm \def\arraystretch{1.2}
       \begin{array}{ll}
          \text{(1) with } b = 3, 5, \ldots, N-1, 
            &\text{(2) with } a = 3, 5, \ldots, N-1, \\
          \text{(5) with } b = 3, 5, \ldots, N-3, 
            &\text{(6) with } a = 3, 5, \ldots, N-3.
       \end{array} 
    \]
    Thus, including the actions of $B_1$ as (3) and $B_2$ as (8), the largest $\theta$ could be is $2 + 2 \left(\frac{N-2} 2 \right) + 2 \left( \frac {N-4} 2 \right) = 2N - 4$.
    On the other hand, if $N$ is odd, we can have $B_i$ acting as
    \[ \arraycolsep=.5cm \def\arraystretch{1.2}
       \begin{array}{ll}
          \text{(1) with } b = 4, 6, \ldots, N-1, 
            &\text{(2) with } a = 4, 6, \ldots, N-1, \\
          \text{(5) with } b = 2, 4, \ldots, N-3, 
            &\text{(6) with } a = 2, 4, \ldots, N-3.
       \end{array} 
    \]
    Hence, the largest $\theta$ could be is $2 + 4 \left(\frac{N-3} 2 \right) = 2N - 4$.
    
    \noindent \textbf{Case 3: $B_1$ acts as (3) with none of the $B_i$ acting as (7) or (8):}
    In this case, any $B_i$ acting as (1), (2), (5), or (6) must satisfy the following, respectively:
    \[
      3 \leq b \leq N-1, \qquad
      3 \leq a \leq N, \qquad
      1 \leq b \leq N-1, \qquad
      2 \leq a \leq N-1.
    \]
    If $N$ is even, we can have $B_i$ acting as 
    \[ \arraycolsep=.5cm \def\arraystretch{1.2}
       \begin{array}{ll}
          \text{(1) with } b = 3, 5, \ldots, N-1, 
            &\text{(2) with } a = 3, 5, \ldots, N-1, \\
          \text{(5) with } b = 1, 3, \ldots, N-1, 
            &\text{(6) with } a = 3, 5, \ldots, N-1.
       \end{array} 
    \]
    Thus, including the action of $B_1$ as (3), the largest $\theta$ could be is $1 + 3 \left(\frac{N-2} 2 \right) + \frac {N} 2 = 2N - 2$.
    On the other hand, if $N$ is odd, we can have $B_i$ acting as
    \[ \arraycolsep=.5cm \def\arraystretch{1.2}
       \begin{array}{ll}
          \text{(1) with } b = 3, 5, \ldots, N-2, 
            &\text{(2) with } a = 3, 5, \ldots, N, \\
          \text{(5) with } b = 1, 3, \ldots, N-2, 
            &\text{(6) with } a = 3, 5, \ldots, N-2.
       \end{array} 
    \]
    Hence, the largest $\theta$ could be is $1 + 2 \left(\frac{N-3} 2 \right) + 2 \left(\frac{N-1} 2 \right) = 2N - 3$.
    
    \noindent \textbf{Case 4: none of the $B_i$ act as (3), (4), (7), or (8):}
    In this case, any $B_i$ acting as (1), (2), (5), or (6) must satisfy the following, respectively:
    \[
      2 \leq b \leq N, \qquad
      2 \leq a \leq N, \qquad
      1 \leq b \leq N-1, \qquad
      1 \leq a \leq N-1.
    \]
    If $N$ is even, we can have $B_i$ acting as 
    \[ \arraycolsep=.5cm \def\arraystretch{1.2}
       \begin{array}{ll}
          \text{(1) with } b = 2, 4, \ldots, N, 
            &\text{(2) with } a = 2, 4, \ldots, N, \\
          \text{(5) with } b = 2, 4, \ldots, N-2, 
            &\text{(6) with } a = 2, 4, \ldots, N-2.
       \end{array} 
    \]
    Thus, the largest $\theta$ could be is $2 \left(\frac{N-2} 2 \right) + 2 \left( \frac {N} 2 \right) = 2N - 2 $.
    On the other hand, if $N$ is odd, we can have $B_i$ acting as
    \[ \arraycolsep=.5cm \def\arraystretch{1.2}
       \begin{array}{ll}
          \text{(1) with } b = 2, 4, \ldots, N-1, 
            &\text{(2) with } a = 2, 4, \ldots, N-1, \\
          \text{(5) with } b = 2, 4, \ldots, N-1, 
            &\text{(6) with } a = 2, 4, \ldots, N-1.
       \end{array} 
    \]
    Hence, the largest $\theta$ could be is $4 \left(\frac{N-1} 2 \right) = 2N - 2$.
    
    Thus, by considering all four cases, we see that if $N$ is even or odd, the maximum that $\theta$ could be is $2N - 2$.
  \end{proof}

\section{Additional results}
\label{sec.add}

In this section, we consider invariants of actions on quantum planes, and actions on further families of algebras related to quantum affine spaces and quantum matrix algebras.

\subsection{Invariants}
We study invariants of some of the actions explored above.
Recall that for a Hopf algebra $H$ and an $H$-module algebra $A$, the ring of invariants is defined as $A^H = \{ a \in A \mid h \cdot a = \varepsilon(h)a \}$.
It is clear that for a generalized Taft algebra $T=T_n(\lambda,m,0)$, $A^T = A^{\grp{g}} \cap A^{\grp{x}}$ and for $B=B(G,\gu,\cu)$ of rank $\theta$, $A^B = \cap_{i=1}^\theta A^{B_i}$.


A connected ($\NN$)-graded algebra $A$ is said to be {\sf Artin-Schelter (AS) regular} if it has finite Gelfand-Kirillov dimension, finite global dimension $d$, $\Ext_A^d(\kk) \iso \kk$, and $\Ext_A^i(\kk)=0$ for $i \neq d$.
Furthermore, a noetherian, regular graded domain $A$ of dimension $d$ with Hilbert series $H_A(t)=(1-t)^{-n}$ is a {\sf quantum polynomial ring}.
It is well-known that the algebras $\cO_q(M_N(\kk))$ and $\kk_\bq[u_1,\hdots,u_N]$ are quantum polynomial rings.

\begin{lemma}
\label{lem.inv1}
Assume $T=T_n(\lambda,m,0)$, $m>2$, acts linearly and inner faithfully on $A=\kk_\bp[u_1,u_2,\cdots,u_t]$ such that $\ord(p_{ij})>2$ for all $i \neq j$ and $x \cdot A\neq 0$.
If $T$ acts as a trivial extension of $A_{12}$, then $A^{\grp{x}} = \kk_{\bp'}[u_1,u_2^m,u_3,\hdots,u_t]$ where $p_{i2}' = p_{i2}^m$ (so $p_{2i}'=p_{2i}^m$ also) and $p_{ij}'=p_{ij}$ when $i \neq 2$ and $j \neq 2$.
\end{lemma}

\begin{proof}
By \cite[Lemma 2.1]{GWY} and Proposition \ref{prop.genaction} we have
\[
x \cdot (u_1^{i_1} u_2^{i_2} \cdots u_t^{i_t})
	= (x \cdot (u_1^{i_1} u_2^{i_2})) u_3^{i_3} \cdots u_t^{i_t}
	= \eta_{12} [i_2]_{\lambda\inv} \alpha_1^{i_1} u_1^{i_1+1} u_2^{i_2-1} u_3^{i_3} \cdots u_t^{i_t}.
\]
Thus, $x \cdot (u_1^{i_1} u_2^{i_2} \cdots u_t^{i_t})=0$ if and only if $i_2 \cong 0 \mod m$ and $A^{\grp{x}}$ is as claimed.
\end{proof}

It is clear that be above lemma generalizes to $T$ acting as a trivial extension of any $A_{ij}$ by a simple change of variable.

Let $A$ be a connected graded algebra and $G$ a finite subgroup of finite automorphisms.
The trace series of $g \in G$ is defined as
\[ \Tr_A(g,t) = \sum \trace\left(\left. g\right|_{A_i} \right) t^i.\]
The trace series was defined by Jing and Zhang \cite{JZ2,JZ1}.
For our purposes, it suffices to recall the following.
Let $(x_1,\hdots,x_n)$ be a normal regular sequence in $A$ such that $A/(x_1,\hdots,x_n)=\kk$
and $g \cdot x_i=\lambda_i x_i$ for all $i=1,\hdots,n$.
By \cite[Lemma 1.7]{KKZ4}, 
\begin{align}
\label{eq.trace}
\Tr_A(g,t) = \left((1-\lambda_1 t^{\deg x_1}) \cdots (1-\lambda_n t^{\deg x_n})\right)\inv.
\end{align}
We apply this along with a version of Molien's Theorem \cite[Lemma 5.1]{JZ1},
\[ H_{A^G}(t) = \frac{1}{\ord(G)} \sum_{g \in G} \Tr_A(g,t).\]
A {\sf reflection} of a quantum polynomial ring $A$ is a graded automorphism $\rho$ such that
\[ \Tr_A(\rho,t) = (1-t)^{-n}(1-\tornado t)\inv.\]
such that $\tornado \neq 1$.
If $A$ is a quantum polynomial ring, then $A^G$ has finite global dimension if and only if $G$ is generated by reflections \cite[Theorem 1.1]{KKZ1}.
The following is a sort of Shephard-Todd-Chevalley Theorem for generalized Taft actions on $\qp$.

\begin{theorem}
\label{thm.gldim}
Let $T=T_n(\lambda,m,0)$ act on $A=\qp$, $\ord(\mu)=k$, according to Proposition \ref{prop.genaction} (a).
Then $A^T$ is commutative. Moreover, $\gldim A^T < \infty$ if and only if $\mu^m=1$.
\end{theorem}

\begin{proof}
As is our convention, we assume $g=\diag(\mu,\lambda\inv \mu)$, $x \cdot u=0$, and $x \cdot v=\eta u$. 
By Proposition \ref{prop.genaction}, $k \mid n$, and by Lemma \ref{lem.inv1}, $A^{\grp{x}}=\kk_{\mu^m}[u,v^m]$. For any monomial in $A^{\grp{x}}$ we have
\begin{align}
\label{eq.gfix}
g \cdot (u^i v^{mj}) = \mu^i (\lambda\inv \mu)^{mj} u^i v^{mj} = \mu^{i+mj} u^i v^{mj}.
\end{align}
Hence, $u^\alpha v^\beta \in A^T$ if and only if $m \mid \beta$ and $k \mid \alpha + \beta$, and all such monomials form a basis for $A^T$.
To see that $A^T$ is always commutative, note that 
$u^\alpha v^\beta u^a v^b = \mu^{\alpha b -\beta a} u^a v^b u^\alpha v^\beta$.
Thus, if $u^\alpha v^\beta, \ u^a v^b \in A^T$, then since $k \mid {\alpha + \beta}$ and $k \mid {a + b}$, we have $k \mid b (\alpha + \beta) - \beta (a + b)$.

By \cite[Theorem 1.1]{KKZ1}, if $A^T = (A^{\grp{x}})^{\grp{g}}$ has finite global dimension, then $\grp{g}$ must be a reflection. 
Note that $g \cdot u=\mu u$ and $g \cdot v^m = (\lambda\inv \mu)^m v^m= \mu^m v^m$. Hence, by \cite[Lemma 1.7]{KKZ4}, 
\[ \Tr_{A^{\grp{x}}}(g,t) = (1-\mu t)\inv(1-\mu^m t)\inv.\]
Since $\mu \neq 1$, then $g$ is a reflection if and only if $\mu^m=1$.
\end{proof}

\begin{corollary}
Suppose $B(G,\gu,\cu)$ has rank $\theta$ and that $B$ acts linearly and inner faithfully on $A=\kk_\bp[u_1,\hdots,u_t]$.
Assume $t$, $m_i$ for all $i$, and $\ord(p_{ij})$ for all $i \neq j$ are at least 3. 
Suppose, for each $i$, $B_i$ acts as a trivial extension of an action on $A_{1i}$. Then $u_1$ belongs to the center of $A^B$.
\end{corollary}

\begin{proof}
Using the argument in Theorem \ref{thm.gldim}, $u_1$ and $u_i$ commute in $A^{B_i}$.
Since $A^B = \cap_i A^{B_i}$ then the result follows.
\end{proof}

Next we determine explicitly the presentation of the fixed ring $A^T$ in certain cases.
Recall that for a graded ring $R$, the {\sf $m$-Veronese subring} is defined as
\[ R_{(m)} = R_0 \oplus R_m \oplus R_{2m} \oplus \cdots.\]

\begin{proposition}
Let $T=T_n(\lambda,m,0)$ act on $A=\qp$, $\ord(\mu)=k$, according to Proposition \ref{prop.genaction} (a).
\begin{enumerate}
\item If $k \mid m$, then $A^T=\kk[u^k,v^m]$.
\item If $m \mid k$, then $A^T = \kk[u^k, u^{k-m}v^m, u^{k-2m}v^{2m}, \ldots, v^k] \cong \kk[a,b]_{\left(\frac{k}{m}\right)}$
\item If $k > m$ and $k-m \mid k$, then $A^T = \kk[u^k, u^{k-m}v^m, v^{\frac{km}{k-m}}] \cong \kk[a,b,c] / (ab - c^{\frac{k}{k-m}})$
\end{enumerate}
\end{proposition}

\begin{proof}
By \eqref{eq.gfix}, we have that $A^T$ has basis $\{u^\alpha v^\beta \mid m \mid \beta, \ k \mid \alpha + \beta\}$.

(1) Suppose $k \mid m$.
Then, assuming $m \mid \beta$, we have $k \mid \alpha + \beta$ if and only if $k \mid \alpha$. Thus, $A^T = \kk[u^k,v^m]$.
This recovers \cite[Lemma 2.1]{GWY} in the case that $m=n$ (when $T$ is a Taft algebra).

(2) Since $m \mid k$, we have that $m \mid \alpha$ for any basis element as well. 
Thus, in this case, $k \mid \alpha + \beta$ if and only if $k/m \mid \alpha/m + \beta/m$.
The basis elements commute, and the isomorphism is given by $u^m \mapsto a$ and $v^m \mapsto b$.

(3) We first show that $A^T$ is generated by $u^k, u^{k-m}v^m, v^{\frac{km}{k-m}}$.
Let $u^\alpha v^\beta \in A^T$.
Without loss of generality, $0 \leq \alpha < k$ and $0 \leq \beta < \frac {km}{k-m}$.
Since $m \mid \beta$, let $\ell = \beta / m$.
Since $k | \alpha + \beta$, we have that $\alpha$ is the unique integer with $0 \leq \alpha < k$ satisfying $\alpha \equiv -\ell m \mod k$.
However, $(k-m)\ell$ is a solution, and $0 \leq (k-m)\ell = \frac{(k-m)\beta}{m} < k$.
Thus, $\alpha = (k-m) \ell$, so $u^\alpha v^\beta = \left( u^{k-m} v^m \right)^\ell$.
Therefore, the generators are as claimed.
Now, by mapping $a \mapsto u^k$, $b \mapsto v^{\frac{km}{k-m}}$, and $c \mapsto u^{k-m}v^m$, we get a surjective homomorphism $\kk[a,b,c] / (ab - c^{\frac{k}{k-m}}) \to A^T$.
We show that this is an isomorphism by considering the Hilbert series.
Note that $\frac k {k-m} - 1 = \frac m {k-m}$ and thus $k-m \mid m$ as well.
Let $s = m / (k-m)$.
Considering $\deg(a) = \deg(c)= k$ and $\deg(b) = sk$, by additivity of the Hilbert series, we have that 
\[
  H_{k[a,b,c]/(ab - c^{\frac{k}{k-m}})} = \frac { H_{k[c]/(c^{\frac{k}{k-m}})} }{(1-t^k) (1-t^{sk})}
    = \frac {1 + t^k + \ldots + t^{sk}} {(1-t^k) (1-t^{sk})}.
\]
On the other hand, $\dim{A^T_{pk}} = p + \lfloor \frac p s  \rfloor + 1$, as we now show.
To see this, let $p = q s + r$ for $0 \leq r < s$, and note that $q = \lfloor \frac p s \rfloor$. 
Then $pk = pm + p(k-m) = (p + q)m + r(k-m)$ and $0 \leq r(k-m) <m$.
Therefore, $A^T_{pk} = \textup{span}_\kk\{u^{pk}, \ u^{pk - m} v^m, \ u^{pk - 2m} v^{2m}, \ldots, u^{r(k-m)} v^{(p+q)m} \}$, giving the desired dimension.
Thus, we have
\[
  H_{A^T} = \sum_{p = 0}^\infty \left( p + \left\lfloor \frac p s \right\rfloor + 1 \right) t^{pk} 
    = \frac {1 + \sum_{p = 1}^\infty \left(1+ \left\lfloor \frac p s \right\rfloor - \left\lfloor \frac {p-1} s \right\rfloor \right) t^{pk} } {(1-t^k)}.
\]
Now $\left\lfloor \frac p s \right\rfloor - \left\lfloor \frac {p-1} s \right\rfloor$ is $1$ if $s \mid p$ and $0$ otherwise.
Thus, we have
\[
  H_{A^T} = \frac {1 + \sum_{p = 1}^\infty t^{pk} + \sum_{q = 1}^\infty t^{qsk} } {(1-t^k)}
    =\frac {1 + t^k + \ldots + t^{sk}} {(1-t^k) (1-t^{sk})}. \qedhere
\]

\end{proof}

\subsection{Quantum exterior algebras}

The Koszul dual of the quantum affine space $\kk_\bp[u_1,\hdots,u_t]$ is the {\sf quantum exterior algebra},
$\bigwedge_\bp(u_1^*,\hdots,u_t^*)$, generated by $u_1^*,\hdots,u_t^*$ subject to the relations
$u_i^* u_j^* + p_{ji} u_j^*u_i^*=0$ and $(u_i^*)^2=0$.
In the case $t=2$, we represent this algebra simply as $\bigwedge_\mu(u^*,v^*)$ where $\mu=p_{12}$.

Suppose $T=T_n(\lambda,m,0)$ acts linearly and inner faithfully on $A=\kk_\mu[u,v]$ according to Proposition \ref{prop.genaction} (a).
Let $A^!=\bigwedge_\mu(u^*,v^*)$ be the Koszul dual of $A$.
Let $S=T_n(\lambda\inv,m,0)$ and set the canonical generators to be $h$ and $y$. There is an action of $S$ on $A^!$ given by
\[ h \cdot u^*=\mu u^*, \quad h \cdot v^*=\lambda\inv \mu v^*, 
\quad y \cdot u^*=\eta v^*, \quad y \cdot v^*=0.\]
It is clear that $h$ is an automorphism of $B$.
We verify below that $y$ acts on $A^!$,
\begin{align*}
y \cdot (u^*)^2 
    &= (h \cdot u^*)(y \cdot u^*)+(y \cdot u^*)u^* 
    = (\mu u^*)(\eta v^*)+(\eta v^*)u^* = 0, \\
y \cdot (v^*)^2 
    &= (h \cdot v^*)(y \cdot v^*)+(y \cdot v^*)v^* = 0, \\
y \cdot (u^*v^*+\mu\inv v^*u^*)
	&= [ (h \cdot u^*)(y \cdot v^*)+(y \cdot u^*)v^*] + \mu\inv [ (h \cdot v^*)(y \cdot u^*)+(y \cdot v^*)u^*] \\
	&=  (\eta v^*)v^* +\mu\inv (\lambda\inv \mu v^*)(\eta v^*) = 0
\end{align*}

\begin{lemma}
\label{lem.exterior}
Suppose $T=T_n(\lambda,m,0)$ acts linearly and inner faithfully on $A=\kk_\bp[u_1,\hdots,u_t]$.
Then there is an action of $S=T_n(\lambda\inv,m,0)$ on $A^!=\Lambda_\bp(u_1^*,\hdots,u_t^*)$
where, as matrices, $h=g$ and $y=x^T$.
\end{lemma}

\begin{proof}
Assume that $T$ acts on $A$ as a trivial extension of an action on $A_{12}$.
The proof for arbitrary $A_{ij}$ is similar.
First note that
\[ 0=(gx-\lambda xg)^T = x^T g^T - \lambda g^T x^T = yh-\lambda hy.\]
It remains to show that $A^!$ is an $S$-module algebra.
By the above argument, $S$ acts on $(A^!)_{12}$.
Suppose $j,k>2$. 
Then clearly 
$y \cdot  (u_k^*)^2 =0$, $y \cdot (u_2^*u_k^*+p_{k2}u_k^*u_2^*) =0$,
and $y \cdot ( u_j^*u_k^* + p_{kj}u_k^*u_j^*)=0$. 
It remains to check that 
$y \cdot ( u_1^*u_k^* + p_{k1}u_k^*u_1^*)=0$.
Note that
\[ 0=x \cdot (u_2u_k-p_{2k}u_ku_2)=u_1u_k-p_{2k}\alpha_ku_ku_1
	= (p_{1k}-p_{2k}\alpha_k)u_ku_1.\]
Hence, $\alpha_k = p_{1k}p_{k2}$. Now
\begin{align*}
y \cdot (u_1^*u_k^* + p_{k1}u_k^*u_1^*) 
	&= (y \cdot u_1^*)u_k^* + p_{k1}(h \cdot u_k^*)(y \cdot u_1^*) \\
	&= (\eta_{12} u_2^*)(u_k^*) + p_{k1}(\alpha_k u_k^*)(\eta_{12} u_2^*) \\
	&= \eta_{12}\left( -p_{k2} + p_{k1}\alpha_k \right)u_k^*u_2^* = 0.
\end{align*}
It follows that $A^!$ is an $S$-module algebra.
\end{proof}

The above proposition extends in a natural way to higher rank actions.
Given some $\gu=\{g_1,\hdots,g_\theta\} \subset G$, set $\gu^T=\{g_1^T,\hdots,g_\theta^T\}$ and similarly for $\cu$.

\begin{proposition}
\label{prop.exterior}
Suppose $B(G,\gu,\cu)$ has rank $\theta$, and that $B$ acts linearly and inner faithfully on $A=\kk_\bp[u_1,\hdots,u_t]$.
Assume $t$, $m_i$ for all $i$, and $\ord(p_{ij})$ for all $i \neq j$ are at least 3. 
Then there is an action of $B'(G,\gu^T,\cu^T)$ on $A^!=\Lambda_\bp(u_1^*,\hdots,u_t^*)$.
\end{proposition}
\begin{proof}
Lemma~\ref{lem.exterior} implies that each $B_i$ acts on $A^!$ as a trivial extension of some $(A^!)_{ij}$ or $(A^!)_{ijk}$.
That $B$ acts on $A$ implies that $B'$ satisfies the necessary compatibility conditions to define an action on $A^!$.
\end{proof}

\subsection{Quantized Weyl algebras}
Let $\bp$ be a multiplicatively antisymmetric $(t \times t)$-matrix and $\gamma=(\gamma_1,\hdots,\gamma_t) \in (\kk^\times)^t$. 
Then $\wanp$ is the algebra with basis $\{u_i,v_i\}$,
$1 \leq i \leq t$, subject to the relations
\begin{align*}
	v_iv_j &= p_{ij} v_jv_i & & (\text{all } i,j) & u_iv_j &= p_{ji} v_ju_i & & (i < j) \\
	u_iu_j &= \gamma_i p_{ij} u_ju_i & & (i < j) & u_iv_j &= \gamma_j p_{ji} v_ju_i & & (i > j) \\
	u_jv_j &= 1 + \gamma_j v_ju_j + \sum_{\ell<j}(\gamma_\ell-1)v_\ell u_\ell & & 
			(\text{all } j).
\end{align*}
The {\sf (multiparameter) quantized Weyl algebra} may be regarded 
as $\gamma$-difference operators on $\kk_\bp[u_1,\hdots,u_t]$. 

We study generalized Taft actions on multiparameter quantum Weyl algebras
that are related to the actions for quantum affines spaces studied previously. Recall that from Proposition \ref{prop.genaction} we understand rank one actions on the first quantum Weyl algebra.

Suppose either the center of $\wanp$ is trivial or is a polynomial ring.
By \cite[Proposition 1.5]{AD} ($t=1$), \cite[Theorem 4.2.5]{R} (trivial center), 
and \cite[Corollary 6.5]{LY} (polynomial ring center), $\phi \in \Aut(\wanp)$ has one of two forms
\begin{enumerate}
\item For all $j \in 1,\hdots,t$, $\phi(u_j)=\alpha_j u_j$ and $\phi(v_j)= \alpha_j\inv v_j$ for some scalars $\alpha_j \in \kk^\times$.
\item Fix $k \in 1,\hdots,t$, then 
  	\begin{itemize}
  	\item $\phi(u_k)=\alpha_k v_k$ and $\phi(v_k)=-\alpha_k\inv u_k$ for some scalar $\alpha_k \in \kk^\times$;
	\item for $j \neq k$, $\phi(u_j)=\alpha_j u_j$ and $\phi(v_j)= \alpha_j\inv v_j$ for some scalars $\alpha_j \in \kk^\times$;
	\end{itemize}
	\end{enumerate}
The second type only occurs under specific conditions on the parameters in the nontrivial center case that we may safely ignore by our hypotheses. The automorphism group in cases not considered above is unknown.

There is a filtration $\cF$ on $\wanp$ defined by setting $\deg(u_i) = \deg(v_i)=i$. 
That is, $\cF=\{F_i\}$ where $F_k = \Span_\kk\{ a \in \wanp \mid \deg(a) \leq k\}$, so that $\wanp = \bigcup_{i \geq 0} F_i$ and $F_iF_j \subset F_{i+j}$.
The associated graded graded ring with respect to this filtration, $\gr_{\cF}(\wanp) = \bigoplus_{i \geq 0} F_i/F_{i-1}$, is a quantum affine space.
An action of a Hopf algebra $H$ {\sf respects the filtration} $\cF$ if $h \cdot F_k \subset F_k$ for all $k$ and all $h \in H$. 

\begin{proposition}
\label{prop.weyl}
Suppose $B(G,\gu,\cu)$ has rank $\theta$, and that $B$ acts inner faithfully on $A=\wanp$ respecting the filtration $\cF$.
Assume the parameters for $B$ and $\gr_\cF A$ are all roots of unity of order at least 3. Then $\rank B \leq 2(2t-1)$.
\end{proposition}

\begin{proof}
Because $B$ respects the filtration on $A$, then the action of $B$ descends to an inner faithful, linear action on $\gr_\cF A$ \cite[Lemma 3.1]{CWWZ}. The result now follows from Theorem \ref{thm.QLSrank}.
\end{proof}

It is not true that every action preserves $\cF$, as illustrated by the next example.

\begin{example}
Consider $A=A_2^{\bp,\gamma}$ and let $T=T_n(\lambda,m,0)$.
Set $\gamma_1=\gamma_2=\lambda$, $\alpha_1=\lambda \alpha_2$, and $\alpha_2=p_{12}$.
We define a diagonal action of $g$ on $A$ as above by setting $g=\diag(\alpha_1,\alpha_1\inv,\alpha_2,\alpha_2\inv)$.
Suppose $x \cdot u_2=u_1$ and $x \cdot v_1=-\alpha_2\inv v_2$, and that $x$ acts as zero on all other generators. It is left to check that $x \cdot r=0$ for all relations $r$. We do the one check below and leave the rest to the reader.
\begin{align*}
x \cdot (u_2v_1-\gamma_1p_{12} v_1u_2)
	&= (u_1v_1 - \gamma_1p_{12} \alpha_1\inv v_1u_1) - \alpha_2\inv ( \alpha_2 u_2v_2 - \gamma_1p_{12} v_2u_2 ) \\
	&= ( 1 + \gamma_1v_1u_1 - v_1u_1) - (1+\gamma_2v_2u_2 + (\gamma_1-1)v_1u_1) + \alpha_2\inv \gamma_1p_{12} v_2u_2 \\
	&= ( 1 + (\gamma_1-1)v_1u_1) - (1 + (\gamma_1-1)v_1u_1) + (-\gamma_2 + \gamma_1)v_2u_2 = 0.
\end{align*}
\end{example}

\subsection{Quantum general and special linear groups}

The {\sf quantum determinant} of $\qmn$ is the central element
\[ \detq = \sum_{\pi \in \cS_N} (-q)^{\ell(\pi)} Y_{1,\pi(1)} Y_{2,\pi(2)} \cdots Y_{n,\pi(n)},\]
where $\ell(\pi)$ denotes the length of the permutation $\pi$. 
In the case $n=2$, this is the element $AD-qBC$.
The quantum determinant can then be used to define the corresponding 
{\sf quantum general linear group} $\qgl=\qmn[\detq\inv]$ and 
{\sf quantum special linear group} $\qsl=\qmn/(\detq-1)$. 

\begin{proposition}
\label{prop.slgl}
Suppose $T_n(\lambda,m,0)$ acts on $\mc O_q(M_N(\kk))$ according to 
\begin{itemize}
    \item rows 1,2,4, or 5 of Table \ref{table.mat2} in the case $N=2$, or
    \item rows 1,2,5, or 6 of Table \ref{table.matn} in the case $N>2$.
\end{itemize}
Then the action descends to an action on $\mc O_q(\SL_N(\kk))$ and lifts to an action on $\mc O_q(\GL_N(\kk))$.
\end{proposition}

\begin{proof}
We claim in the cases listed above that both $(\det_q-1)$ and $(\detq)$ are $T$-stable ideals. Hence, the action descends to an action on $\mc O_q(M_N(\kk))/(\detq-1)$ and lifts to an action on $\mc O_q(M_N(\kk))[\detq\inv]$ \cite[Corollary 3.14]{MS}.
Since $\detq$ is central, it suffices to prove for both cases that $g \cdot \detq=\det_q$ and $x \cdot \det_q=0$.

Fix $b>1$ and suppose $\delta \neq 0$.
We prove this for the case that $g \cdot Y_{a,b}=q\inv Y_{a,b}$, $g \cdot Y_{a,b-1}=qY_{a,b-1}$, $x \cdot Y_{a,b}=\delta Y_{a,b-1}$ and $x \cdot Y_{r,s}=0$ for all $a$ and for all $r,s$ with $s \neq b,b-1$. This corresponds to row 1 in either Table \ref{table.mat2} or \ref{table.matn}.
A similar argument applies to the other cases.

It is clear that $g \cdot \detq=\detq$. We will prove the result for $x$ using induction and the {\it quantum Laplace expansion} \cite[Corollary 4.4.4]{PW}.
First suppose that $N=2$ and consider the action defined in the first row of Table \ref{table.mat2}, which corresponds to the action above. Recall that the quantum determinant in the case of $\mc O_q(M_2(\kk))$ is $\detq=AD-qBC$. Hence,
\[ x \cdot \detq = \left[ (qA)(\delta C) + 0 \right] - q\left[ 0+(\delta A)C\right] = 0.\]

Now suppose $N \geq 3$ and fix $i \neq b,b-1$.
Expanding along the $i$th column, we have
\[ \detq = \sum_{k=1}^N (-q)^{i-k} A_{ki} Y_{ki},\]
where $A_{ki}$ is the $(k,i)$-quantum minor of $\mc O_q(M_N(\kk))$. By induction, $x \cdot A_{ki}=0$, and since $x \cdot Y_{ki}=0$, then it follows that $x \cdot \detq=0$.
\end{proof}

\section{Questions and remarks}
\label{sec.ques}

Bosonizations of quantum linear spaces are important examples of finite-dimensional pointed Hopf algebras, however the full classification is much more robust \cite{AS1}. Moreover, we placed restrictions on our parameters that may ultimately be artificial. (See Proposition \ref{prop.genaction}.) It would be interesting to know how the bounds presented in this paper fit into the story of more general actions.

\begin{question}
Let $H$ be a finite-dimensional pointed Hopf algebra acting linearly and inner faithfully on a quantum affine space or quantum matrix algebra.
Does $\rank(H)$ satisfy the same bounds as in Theorems \ref{thm.QLSrank}, \ref{thm.mat2}, or \ref{thm.matn}?
\end{question}

On the other side, there are many important families of quantum algebras for which we have not or have only partially considered the problem of classifying actions.
In Section \ref{sec.add}, we classified induced actions on certain families of algebras. This leads naturally to a question of whether a more full classification is possible.
Of these, actions on quantum exterior algebras and quantized Weyl algebras seem the most within reach.

\begin{question}
Do the bounds in Theorems \ref{thm.QLSrank}, \ref{thm.mat2}, or \ref{thm.matn} apply also to the ``related'' algebras considered in Section \ref{sec.add}?
\end{question}

The quantum matrix algebras considered in this paper are the single-parameter versions of a larger class of {\it multiparameter} quantized matrix algebras (see \cite{BGquantum}). It was clear to us that the classification problem for generalized Taft algebras in this case is substantially more difficult. Nevertheless, under suitable restrictions, it seems reasonable that one could attack this problem with some level of success.

\begin{question}
Does the classification of generalized Taft algebra actions on multiparameter quantized matrix algebras align with the single-parameter versions given in Propositions \ref{prop:matrixactions} and \ref{prop:matrixactions2}? Do the bounds given in Theorems \ref{thm.mat2} and \ref{thm.matn} still apply?
\end{question}

More generally, we wonder whether there are methods that can simplify or consolidate some of the computations above.

\begin{question}
Are there algebra invariants that control actions of pointed Hopf algebras? Locally nilpotent derivations are controlled by the (noncommutative) discriminant \cite{CPWZ1}. Is there an analogue for skew derivations associated to generalized Taft algebras?
\end{question}

In Section \ref{sec.add} we studied invariants of actions under generalized Taft actions. However, we were only able to determine properties and the form of the fixed ring in certain cases.

\begin{question}
In general, is there is a nice presentation of the fixed ring $\kk_\mu[u,v]$ under a generalized Taft action? When does the fixed ring of $\kk_\bp[u_1,\hdots,u_t]$ under a generalized Taft action have finite global dimension?
\end{question}

\subsection*{Acknowledgment}
The second author was partially supported by a grant from the Miami University Senate Committee on Faculty Research.
The authors thank Chelsea Walton for helpful conversations. 


\begin{thebibliography}{10}

\bibitem{AC}
J.~Alev and M.~Chamarie.
\newblock D\'erivations et automorphismes de quelques alg\`ebres quantiques.
\newblock {\em Comm. Algebra}, 20(6):1787--1802, 1992.

\bibitem{AD}
J.~Alev and F.~Dumas.
\newblock Rigidit\'e des plongements des quotients primitifs minimaux de
  {$U_q({\rm sl}(2))$} dans l'alg\`ebre quantique de {W}eyl-{H}ayashi.
\newblock {\em Nagoya Math. J.}, 143:119--146, 1996.

\bibitem{AS2}
N.~Andruskiewitsch and H.-J. Schneider.
\newblock Lifting of quantum linear spaces and pointed {H}opf algebras of order
  {$p^3$}.
\newblock {\em J. Algebra}, 209(2):658--691, 1998.

\bibitem{AS1}
Nicol\'as Andruskiewitsch and Hans-J\"urgen Schneider.
\newblock On the classification of finite-dimensional pointed {H}opf algebras.
\newblock {\em Ann. of Math. (2)}, 171(1):375--417, 2010.

\bibitem{BGquantum}
Ken~A. Brown and Ken~R. Goodearl.
\newblock {\em Lectures on algebraic quantum groups}.
\newblock Advanced Courses in Mathematics. CRM Barcelona. Birkh\"{a}user
  Verlag, Basel, 2002.

\bibitem{CPWZ1}
S.~Ceken, J.~H. Palmieri, Y.-H. Wang, and J.J. Zhang.
\newblock The discriminant controls automorphism groups of noncommutative
  algebras.
\newblock {\em Adv. Math.}, 269:551--584, 2015.

\bibitem{CPWZ2}
S.~Ceken, J.~H. Palmieri, Y.-H. Wang, and J.J. Zhang.
\newblock The discriminant criterion and automorphism groups of quantized
  algebras.
\newblock {\em Adv. Math.}, 286:754--801, 2016.

\bibitem{CWWZ}
K.~Chan, C.~Walton, Y.~H. Wang, and J.~J. Zhang.
\newblock Hopf actions on filtered regular algebras.
\newblock {\em J. Algebra}, 397:68--90, 2014.

\bibitem{CKWZ}
Kenneth Chan, Ellen Kirkman, Chelsea Walton, and James~J. Zhang.
\newblock Quantum binary polyhedral groups and their actions on quantum planes.
\newblock {\em J. Reine Angew. Math.}, 719:211--252, 2016.

\bibitem{cliff}
Gerald Cliff.
\newblock The division ring of quotients of the coordinate ring of the quantum
  general linear group.
\newblock {\em J. London Math. Soc. (2)}, 51(3):503--513, 1995.

\bibitem{Cl}
Zachary Cline.
\newblock On actions of {D}rinfel'd doubles on finite dimensional algebras.
\newblock {\em J. Pure Appl. Algebra}, 223(8):3635--3664, 2019.

\bibitem{EW2}
Pavel Etingof and Chelsea Walton.
\newblock Finite dimensional {H}opf actions on algebraic quantizations.
\newblock {\em Algebra Number Theory}, 10(10):2287--2310, 2016.

\bibitem{EW}
Pavel Etingof and Chelsea Walton.
\newblock Pointed {H}opf actions on fields, {II}.
\newblock {\em J. Algebra}, 460:253--283, 2016.

\bibitem{GWY}
Jason Gaddis, Robert Won, and Daniel Yee.
\newblock Discriminants of {T}aft {A}lgebra {S}mash {P}roducts and
  {A}pplications.
\newblock {\em Algebr. Represent. Theory}, 22(4):785--799, 2019.

\bibitem{JZ2}
N.~Jing and J.~J. Zhang.
\newblock Gorensteinness of invariant subrings of quantum algebras.
\newblock {\em J. Algebra}, 221(2):669--691, 1999.

\bibitem{JZ1}
Naihuan Jing and James~J. Zhang.
\newblock On the trace of graded automorphisms.
\newblock {\em J. Algebra}, 189(2):353--376, 1997.

\bibitem{KW}
Ryan Kinser and Chelsea Walton.
\newblock Actions of some pointed {H}opf algebras on path algebras of quivers.
\newblock {\em Algebra Number Theory}, 10(1):117--154, 2016.

\bibitem{KKZ1}
E.~Kirkman, J.~Kuzmanovich, and J.~J. Zhang.
\newblock Shephard-{T}odd-{C}hevalley theorem for skew polynomial rings.
\newblock {\em Algebr. Represent. Theory}, 13(2):127--158, 2010.

\bibitem{KKZ4}
E.~Kirkman, J.~Kuzmanovich, and J.~J. Zhang.
\newblock Invariant theory of finite group actions on down-up algebras.
\newblock {\em Transform. Groups}, 20(1):113--165, 2015.

\bibitem{KR}
Leonid Krop and David~E. Radford.
\newblock Finite-dimensional {H}opf algebras of rank one in characteristic
  zero.
\newblock {\em J. Algebra}, 302(1):214--230, 2006.

\bibitem{ll-aoqm}
S.~Launois and T.~H. Lenagan.
\newblock Automorphisms of quantum matrices.
\newblock {\em Glasg. Math. J.}, 55(A):89--100, 2013.

\bibitem{LY}
Jesse Levitt and Milen Yakimov.
\newblock Quantized {W}eyl algebras at roots of unity.
\newblock {\em Israel J. Math.}, 225(2):681--719, 2018.

\bibitem{MS}
S.~Montgomery and H.-J. Schneider.
\newblock Hopf crossed products, rings of quotients, and prime ideals.
\newblock {\em Adv. Math.}, 112(1):1--55, 1995.

\bibitem{PW}
Brian Parshall and Jian~Pan Wang.
\newblock Quantum linear groups.
\newblock {\em Mem. Amer. Math. Soc.}, 89(439):vi+157, 1991.

\bibitem{R}
Laurent Rigal.
\newblock Spectre de l'alg\`ebre de {W}eyl quantique.
\newblock {\em Beitr\"age Algebra Geom.}, 37(1):119--148, 1996.

\bibitem{T}
Earl~J. Taft.
\newblock The order of the antipode of finite-dimensional {H}opf algebra.
\newblock {\em Proc. Nat. Acad. Sci. U.S.A.}, 68:2631--2633, 1971.

\bibitem{y-llc}
Milen Yakimov.
\newblock The {L}aunois-{L}enagan conjecture.
\newblock {\em J. Algebra}, 392:1--9, 2013.

\end{thebibliography}
\bibliographystyle{plain}

\end{document}